\documentclass[reqno,a4paper, 11pt]{amsart}

\usepackage[latin1]{inputenc}
\usepackage[italian,english]{babel}
\usepackage{eucal,amsfonts,amssymb,amsmath,amsthm,epsfig,mathrsfs}
\usepackage{dsfont,cancel,soul}
\usepackage{color}
\usepackage{amscd,amsxtra}
\usepackage{enumerate}
\usepackage{latexsym}
\usepackage{bm}
\usepackage{comment}

\allowdisplaybreaks

\newcounter{ipotesi}

\Alph{ipotesi}
 \makeatletter \@addtoreset{equation}{section}

\makeatother \makeatletter

\newtheorem{thm}{Theorem}[section]
\newtheorem{hyp}[thm]{Hypotheses}{\rm}
\newtheorem{hyp0}[thm]{Hypothesis}{\rm}
\newtheorem{lemm}[thm]{Lemma}
\newtheorem{coro}[thm]{Corollary}
\newtheorem{prop}[thm]{Proposition}

\newtheorem{rmk}[thm]{Remark}{\rm}
\newtheorem{example}[thm]{Example}

\newcounter{parentenv}

\newcommand{\R}{{\mathbb R}}

\newcommand{\N}{{\mathbb N}}

\newcommand{\Rd}{\mathbb R^d}
\newcommand{\Rn}{\mathbb R^n}
\newcommand{\Rm}{\mathbb R^m}

\newcommand{\T}{{\bm T}}

\newcommand{\g}{{\bm g}}
\newcommand{\f}{\bm f}
\newcommand{\uu}{{\bm u}}
\newcommand{\A}{\bm{\mathcal A}}
\newcommand{\1}{\mathds 1}

\newcommand{\vv}{{\bm v}}
\newcommand{\ww}{{\bm w}}

\newcommand{\G}{{\bm G}}
\newcommand{\RR}{{\bm R}}
\newcommand{\B}{{\bm{\mathcal B}}}

\begin{document}

\title[On weakly coupled systems]{On weakly coupled systems of partial differential equations with different diffusion terms}
\thanks{The authors are members of G.N.A.M.P.A. of the Italian Istituto Nazionale di Alta Matematica (INdAM).}
\author[D. Addona, L. Lorenzi]{Davide Addona, Luca Lorenzi}
\address{Dipartimento di Scienze Matematiche, Fisiche e Informatiche, Plesso di Matematica, Universit\`a di Parma, Parco Area delle Scienze 53/A, I-43124 PARMA (Italy)}
\email{davide.addona@unipr.it; https://orcid.org/0000-0002-6372-0334}
\email{luca.lorenzi@unipr.it; https://orcid.org/0000-0001-6276-5779}

\date{}


\begin{abstract}
We prove maximal Schauder regularity for solutions to elliptic systems and Cauchy problems, in the space $C_b(\Rd;\Rm)$ of bounded and continuous functions, associated to a class of nonautonomous weakly coupled second-order elliptic operators $\A$, with possibly unbounded coefficients and diffusion and drift terms which vary from equation to equation. We also provide estimates of the spatial derivatives up to the third-order and continuity properties both of the evolution operator $\G(t,s)$ associated to the Cauchy problem $D_t\uu=\A(t)\uu$ in $C_b(\Rd;\Rm)$, and, for fixed $\overline t$, of the semigroup $\T_{\overline t}(\tau)$ associated to the autonomous Cauchy problem $D_\tau\uu=\A(\overline t)\uu$ in $C_b(\Rd;\Rm)$. These results allow us to deal with elliptic problems whose coefficients also depend on time.
\end{abstract}
\maketitle

\section{Introduction}

In this paper, we consider weakly coupled nonautonomous elliptic operators $\bm {\mathcal A}$, defined on
smooth functions $\boldsymbol\zeta:\Rd\to\Rm$ ($m\ge 2$), by
\begin{align}\label{picco}
(\A(t)\boldsymbol\zeta(x))_k& =\sum_{i,j=1}^dq_{ij}^k(t,x)D_{ij}\zeta_k(x)+
\sum_{j=1}^db^k_j(t,x)D_j\zeta_k(x)+\sum_{h=1}^mc_{kh}(t,x)\zeta_h(x)\notag\\
& = {\rm Tr}(Q^k(t,x)D^2\zeta_k(x))+ \langle {\bm b}^k(t,x), \nabla \zeta_k(x)\rangle+ (C(t,x)\boldsymbol\zeta(x))_k,
\end{align}
$k=1,\ldots,m$, for every $(t,x)\in I\times \Rd$, with possible unbounded coefficients, and for every $T>s\in I$ (where $I$ is a right halfline, possibily $I=\R$) we prove maximal Schauder regularity for the solutions to the elliptic system
\begin{align}
\label{intro:ell_eq}
\lambda \uu(t,x)-\A(t)\uu(t,x)=\f(t,x), \quad t\in[s,T], \ x\in\Rd,
\end{align}
and to the Cauchy problem
\begin{align}
\label{intro:cau_pro}
\left\{
\begin{array}{lll}
D_t\uu(t,x)=\A(t)\uu(t,x)+\g(t,x), & t\in[s,T], & x\in\Rd, \vspace{1mm} \\
\uu(s,x)=\f(x), & & x\in\Rd.
\end{array}
\right.
\end{align}

The study of scalar equations with unbounded coefficients has been widely studied in the last two decades and the theory is well-established; we refer to the monograph \cite{Lo17} and the references therein for a systematic treatment.
On the other hand, the study of systems of Kolmogorov equations with unbounded coefficients is at the beginning and just few results are available.

One of the first papers dealing with systems of parabolic equations with unbounded coefficients is \cite{HLRS09}, where the authors study the realization of the weakly coupled elliptic operator $\A_p\uu:={\rm div} (Q\nabla \uu)+\langle F,\nabla \uu\rangle +V\uu$ in $L^p(\Rd;\Rm)$ and characterize its domain under suitable assumptions on the coefficients. Weakly coupled operators in the space of bounded and continuous functions $C_b(\Rd;\Rm)$ have been considered in \cite{DeLo11}, where the authors prove existence and uniqueness of a classical solution to the associated Cauchy problem. This allows to define a semigroup of bounded operators $(\T(t))_{t\geq0}$ on $C_b(\Rd;\Rm)$, and then to study some of its main properties, such as compactness, uniform estimates of the derivatives and optimal Schauder regularity.

A first improvement of the results in \cite{DeLo11} appears in \cite{AALT17}, where nonautonomous second-order operators coupled up to first order are considered. Also in this case, existence and uniqueness of a classical solution to the Cauchy problem associated with the operator $\A(t)$ allow the authors to define an evolution operator $(\G(t,s))_{t\geq s}$ on $C_b(\Rd;\Rm)$ and to investigate the main features of the evolution operator. In particular, thanks to weighted gradient estimates the authors prove the existence of Nash equilibria for a class of stochastic differential games.
The case of operators coupled up to the first-order has been considered also in \cite{AngLorPal16}, where the authors provide sufficient conditions for the evolution operator $(\G(t,s))_{t\geq s}$ to be extended to the space $L^p(\Rd;\Rm)$, and some improving summability properties of the evolution operator are shown. 

We stress that in all the quoted paper the diffusion coefficients are the same for each equation, i.e., the operator $\A(t)$ is defined on smooth functions $\f:\Rd\rightarrow \Rm$ by
\begin{align*}
\left((\A(t)\f)(x)\right)_k:=\sum_{i,j=1}^d q_{ij}(t,x) D_{ij}f_k(x)+\sum_{j=1}^d \sum_{h=1}^mb_{jh}^k(t,x)D_jf_h(x)+\sum_{h=1}^m c_{kh}(t,x)f_h(x),
\end{align*}
for every $k=1,\ldots,m$ and every $(t,x)\in I\times \Rd$. Under suitable conditions on the coefficients of the operator $\A$, this allows to apply a maximum principle, and so uniqueness of the classical solution to the Cauchy problem associated with $\A$ follows. We notice that this is the main effort when one deals with systems of Kolmogorov equations with unbounded coefficients, since existence of a classical solution usually follows by means of an approximation procedure which mimics the same technique in the case of a single equation (see e.g., \cite{MePaWa02}) and relies on a-priori (interior Schauder) estimates.

The first attempt to deal with operators with unbounded coefficients and diffusion terms which may vary from line to line appears in \cite{AngLor20}, where the authors study nonautonomous second-order weakly coupled operators defined as \eqref{picco}. Thanks to a generalization of a result in \cite{Pr67}, the authors get existence and uniqueness of a classical solution to the Cauchy problem associated with $\A$, and define an evolution operator $(\G(t,s))_{t\geq s}$ by means of this solution. Some remarkable properties of the evolution operator, such as its compactness, its action on some functional spaces and the existence of systems of family of invariant measures, are investigated. 
The approach in \cite{AngLor20} strongly relies on the assumption that the off-diagonal entries of the matrix $C(t,\cdot)$ are bounded from below for every $t\in I$.
In this paper we are able to remove this assumption thanks to a comparison principle, which allows us to ``control'' the solution of the Cauchy problem associated to $\A$ with that associated to a suitable nonautonomous operator $\A^P$ which satisfies the assumptions in \cite{AngLor20}.
Such an operator $\A^P$ is a fundamental auxiliary tool in the paper: it also appears in Section \ref{section:grad_est}, in order to get both the estimates for spatial derivatives of the evolution operator and the maximal Schauder regularity for problem \eqref{intro:ell_eq}, and in Section \ref{section:special_case} to extend some results in \cite{AdAnLo19} in the autonomous case, showing that the semigroup associated with a class of autonomous systems of ellipic operators extends to a strongly continuous semigroup on the $L^p$-space with respect to a system of {\em signed} invariant measures $\boldsymbol \mu=(\mu_1,\ldots,\mu_m)$ associated to the semigroup. An analogous result has been already proved in \cite{AdAnLo19}, but under the assumptions that the semigroup is nonnegative, and so each component of $\boldsymbol \mu$ is a positive measure, while, in our situation, the semigroup does not preserve positivity and we allow some component of $\boldsymbol \mu$ to be negative. To the best of our knowledge, this is the first time when a semigroup associated to a system of elliptic equations is extended to the $L^p$-spaces with respect to a system of invariant measures, without any positivity preserving condition on it. Also in \cite{AdAnLo18} systems of invariant measures for semigroups associated with a class of autonomous systems of elliptic operators, which do not preserve positivity, are studied, but no results on the possibility to extend the semigroups to $L^p$-spaces related to such systems of measures are established. We finally remark that an abstract approach to asymptotic behaviour of semigroups can be found in \cite{GeGlKu20}, but again under some positivity assumption on the semigroup. 

The paper is organized as follows.
In Section \ref{section:preliminaries}, we provide the main hypotheses on the coefficients of the operator $\A$ and fix some notation which will be useful in the rest of the paper. Further, we prove a comparison principle between the solutions of Cauchy problems on balls, with homogeneous Neumann boundary conditions, associated to $\A$ and $\A^P$, and, based on this result, we show existence, uniqueness and some continuity properties of the evolution operator $(\G(t,s))_{t\geq s}$ (simply denoted by $\G(t,s)$ from now on) associated to the nonautonomous equation $D_t\uu=\A\uu$ on $C_b(\Rd;\Rm)$. Finally, we deduce analogous results for the semigroup $(\T_{\overline t}(\tau))_{\tau\geq0}$ (simply denote by $\T_{\overline t}(\tau)$ from now on) associated to the autonomous equation $D_\tau\uu=\A(\overline t)\uu$ on $C_b(\Rd;\Rm)$,  where $\overline t\in I$ is fixed.

In Section \ref{section:grad_est}, under additional assumptions on the coefficients of the operator $\A$, we prove some estimates for the derivatives of the functions $\G(t,s)\f$ and $\T_{\overline t}(\tau)\f$. To be more precise, we show that for every integer numbers $0\leq h\leq k\leq 3$, every $[s,T]\subset I$, and for every $\overline t\in I$ and every $\mathcal T>0$ there exist positive constants $C$ and $K$ such that
\begin{align*}
\|\G(t,s)\f\|_{C_b^k(\Rd;\Rm)}& \leq \frac{Ce^{K(t-s)}}{(t-s)^{(k-h)/2}}\|\f\|_{C^h_b(\Rd;\Rm)}, \qquad\;\, \f\in C^h_b(\Rd;\Rm),\;\, t\in(s,T]. \\
\|\T_{\overline t}(\tau)\f\|_{C_b^k(\Rd;\Rm)}& \leq \frac{Ce^{K\tau}}{\tau^{(k-h)/2}}\|\f\|_{C^h_b(\Rd;\Rm)}, \qquad\;\, \f\in C^h_b(\Rd;\Rm),\;\, \tau\in (0,\mathcal T].
\end{align*}

These estimates are the starting point to prove, in Section \ref{sec:max_reg}, maximal Schauder regularity for problems \eqref{intro:ell_eq} and \eqref{intro:cau_pro}. 

In Section \ref{section:special_case}, assuming that the diffusion and drift coefficients are the same for each equation and the operator $\A$ is autonomous, we prove the existence of a system of invariant measures for the semigroup $\T(t)$ associated to the Cauchy problem $D_t\uu=\A\uu$,
characterizing them. Moreover, we show that such a semigroup extends to a strongly continuous semigroup on $L^p_{\boldsymbol \mu}(\Rd;\Rm)$ 
and  study the asymptotic behaviour of the semigroup $\T(t)$ as $t$ tends to $\infty$.

\medskip
\paragraph{\bf Notation}
Functions with values in $\R^m$ are displayed in bold style. Given a function $\f$ (resp. a sequence
$(\f_n)$) as above, we denote by $f_i$ (resp. $f_{n,i}$) its $i$-th component (resp. the $i$-th component
of the function $\f_n$). We denote by $|\f|$ the vector-valued function whose components are $|f_1|,\ldots,|f_m|$.
By $B_b(\Rd;\Rm)$ we denote the set of all the bounded Borel measurable functions $\f:\Rd\to\Rm$, where $\|\f\|_{\infty}^2=\sum_{k=1}^m\sup_{x\in\Rd}|f_k(x)|^2$.
For every $k\ge 0$, $C^k_b(\R^d;\R^m)$ is the space of all the functions
whose components belong to $C^k_b(\R^d)$, where ``$b$'' stays for bounded. Similarly, we use the subscript ``$c$'' for spaces of functions with compact support.
When $k\in (0,1)$ we write $C^k_{\rm loc}(\Rd)$ to denote the space of all the real-valued functions $f\in C(\Rd)$
which are H\"older continuous in every compact set of $\Rd$.
We assume that the reader is familiar with the parabolic spaces $C^{h+\alpha/2,k+\alpha}(I\times \Rd)$
($\alpha\in [0,1)$, $h,k\in\N\cup\{0\}$), and we use the subscript ``loc'' with the same meaning as above.
The symbols $D_tf$, $D_i f$, $D_{ij}f$ and $D_{ijh}f$, respectively, denote the time derivative $\frac{\partial f}{\partial t}$ and the spatial derivatives $\frac{\partial f}{\partial x_i}$, $\frac{\partial^2f}{\partial x_i\partial x_j}$
and $\frac{\partial^3f}{\partial x_i\partial x_j\partial x_h}$ for every $i,j,h=1, \ldots,d$.
For $k=1,2,3$, $|D_x^k\f|^2$ denotes the sum $\sum_{j=1}^m|D^k_xf_j|^2$,
where $|D_xf_j|^2=\sum_{i=1}^m|D_if_j|^2$,
$|D_x^2f_j|^2=\sum_{i,h=1}^m|D_{ih}f_j|^2$
and $|D_x^3f_j|^2=\sum_{i,h,k=1}^m|D_{ihk}f_j|^2$. If $k=1$ we write $D_x\f$ and $J_x\f$ indifferently for the Jacobian matrix of $\f$ with respect to the spatial variables. If $\f$ does not depend on $t$, we omit the subscript $x$  and we simply write $D^k\f$.
Similarly, we write $D^2_x\f$ and $D^3_x\f$, when $\f$ depends also on $t$, to denote differentiation with respect to the spatial variables.

By $e_j$ and $\1$ we denote, respectively, the $j$-th vector of the Euclidean basis of $\R^d$ and the function identically equal to $1$ in $\Rm$.
The open ball in $\Rd$ centered at $ 0$ with radius $r>0$ and its closure are denoted by  $B_r$ and $\overline B_r$, respectively.

\section{Hypotheses and preliminary results}
\label{section:preliminaries}
Besides the nonautonomous operator $\A$ defined in \eqref{picco}, we introduce the auxiliary nonautonomous operator ${\bm\A}^P$, defined on smooth functions $\f:\R^d\rightarrow \R^m$ by
\begin{align}
\label{auxiliary_operator}
(({\A}^P(t)\f)(x))_k:=\sum_{i,j=1}^nq^k_{ij}(t,x)D_{ij}f_k(x)+\sum_{j=1}^nb^k_j(t,x)D_jf_k(x)+(C^P(t,x)\f(x))_k
\end{align}
for every $k=1,\ldots,m$ and $(t,x)\in I\times \Rd$, where the matrix $C^P=(c_{hk}^P)$ is given by
\begin{align*}
c_{kh}^P:=
\begin{cases}
c_{kh}, & h=k, \\
|c_{kh}|, & h\neq k,
\end{cases}, \quad h,k=1,\ldots,m.
\end{align*}

Throughout the paper, if not otherwise specified, we assume the following assumptions on the coefficients of the operator $\bm{\mathcal A}$ and ${\A^P}$ in \eqref{picco} and in \eqref{auxiliary_operator}.

\begin{hyp}
\label{hyp-base}
\begin{enumerate}[\rm (i)]
\item
The coefficients $q^k_{ij}=q^k_{ji}$, $b^k_j$ and the entries $c_{kh}$ of the matrix-valued function $C$ belong to $C^{\alpha/2,\alpha}_{\rm loc}(I\times \Rd)$ for some $\alpha\in (0,1)$;
\item
for every bounded interval $J\subset I$, the infimum $\mu_0$ over $J\times \Rd$ of the minimum eigenvalue $\mu^k(t,x)$ of the matrix $Q^k(t,x)=(q^k_{ij}(t,x))$ is positive for every $k=1,\ldots,m$;
\item
for every bounded interval $J\subset I$ there exists a componentwise positive function $\boldsymbol\varphi_J\in C^2(\Rd;\Rm)$, blowing up componentwise as $|x|$ tends to $\infty$
such that ${\A^P}\boldsymbol\varphi_J\le \lambda_J\boldsymbol\varphi_J$ in $J\times\Rd$, for some positive constant $\lambda_J$;
\item
there does not exist a nontrivial set $K\subset \{1, \ldots, m\}$ such that the coefficients $c^P_{ij}$ identically vanish on $I\times \Rd$ for every $i\in K$ and $j\notin K$;
\item
for every bounded interval $J\subset I$, the sum of the elements of each row of $C^P$ is a bounded from above function on $J\times \Rd$.
\end{enumerate}
\end{hyp}

\begin{rmk}
{\rm 
Hypothesis $\ref{hyp-base}(iv)$ is a condition on the entries of the matrix-valued function $C^P$ which guarantees that both the differential systems \eqref{pioggia} and \eqref{cauchy_prob_pos} cannot be (partially) decoupled.
It is not hard to see that if $\A$ satisfies \cite[Hypotheses 2.1]{AngLor20}, then the operators $\A$ and $\A^P$ satisfy Hypotheses \ref{hyp-base}} above.
\end{rmk}

For every bounded interval $J\subset  I$ we set
\begin{align}
\label{def_M_T}
M_{J}:=\max_{k=1,\ldots,m}\sup_{(t,x)\in J\times \Rd}\sum_{j=1}^mc_{kj}^P(t,x)=:\max_{k=1,\ldots,m}\sup_{(t,x)\in J\times \Rd}M_k(t,x).
\end{align}

For every $s,t\in I$, every $\alpha\in[0,1)$ and every bounded open set $\Omega\subset \Rd$ we denote by $\|\A(t)\|_{\alpha,\Omega}$ and by $\|\A(t)-\A(s)\|_{\alpha,\Omega}$ the following quantities:
\begin{align*}
\|\A(t)\|_{\alpha,\Omega}:=& \max_{
{i,j=1,\ldots,d}\atop{h,k=1,\ldots,m}}
\left\{\|q^k_{ij}(t,\cdot)\|_{C_b^\alpha(\Omega)}, \|b^k_{i}(t,\cdot)\|_{C_b^\alpha(\Omega)}, \|c_{kh}(t,\cdot)\|_{C_b^\alpha(\Omega)}\right\}, \\
\|\A(t)-\A(s)\|_{\alpha,\Omega}:=& \max_{{i,j=1,\ldots,d}\atop{ h,k=1,\ldots,m}}\left\{\|q^k_{ij}(t,\cdot)-q_{ij}^k(s,\cdot)\|_{C_b^\alpha(\Omega)}, \|b^k_{i}(t,\cdot)-b_i^k(s,\cdot)\|_{C_b^\alpha(\Omega)},  \right.\\
& \qquad\qquad\;\;  \left.\|c_{kh}(t,\cdot)-c_{kh}(s,\cdot)\|_{C_b^\alpha(\Omega)}\right\},
\end{align*}

\paragraph{\bf {Further notation}.}
Let us fix $R>0$.
\begin{enumerate}[(i)]
\item 
We denote by $\G_{R}^{\mathcal N}(t,s)$ (resp. $\G_{R}^{\mathcal N,P}(t,s)$) the evolution operator associated to the realization of the nonautonomous operator $\A$ (resp. $\A^p$) in $C(\overline{B_R};\Rm)$ with homogeneous Neumann boundary conditions. 
\item 
For every fixed $\overline t\in I$ we denote by $\T_{\overline t}^{\mathcal N,R}(\tau)$ (resp. $\T_{\overline t}^{\mathcal N,R,P}(\tau)$) the semigroup associated to the realization of the operator $\A(\overline t)$ (resp. $\A^P(\overline t)$) in $C(\overline{B_R};\Rm)$ with homogeneous Neumann boundary conditions.
\end{enumerate}
We notice that the existence of the above families of operators is guaranteed by \cite[Section IV.2, Theorem 3.4]{Ei69}.

Fix $s\in I$ and consider the Cauchy problems
\begin{equation}
\left\{
\begin{array}{lll}
D_t\uu(t,x)=(\A(t)\uu)(t,x), & t\in(s,\infty), &x\in\Rd,\\[1mm]
\uu(s,x)=\f(x), && x \in \Rd,
\end{array}
\right.
\label{pioggia}
\end{equation}
and
\begin{equation}
\left\{
\begin{array}{lll}
D_t\uu(t,x)=(\A^P(t)\uu)(t,x), & t\in(s,\infty), &x\in\Rd,\\[1mm]
\uu(s,x)=\f(x), && x \in \Rd.
\end{array}
\right.
\label{cauchy_prob_pos}
\end{equation}
By \cite[Theorem 2.4]{AngLor20} problem \eqref{cauchy_prob_pos} admits a unique classical solution $\uu^P$, which is bounded in each strip $[s,T]\times \Rd\subset I\times\Rd$, belongs to $ C^{1+\alpha/2,2+\alpha}_{\rm loc}((s,\infty)\times\Rd;\Rm)\cap C([s,\infty)\times\Rd;\Rm)$ and satisfies the estimate
$\|\uu^P(t,\cdot)\|_{\infty}\le e^{K_{[s,T]}(t-s)}\|\f\|_{\infty}$ for every $t\in[s,T]$, where $K_{[s,T]}$ is explicitly computed.

The aim of this section is to prove that for every $\f\in C_b(\Rd;\R^m)$ the Cauchy problem \eqref{pioggia} admits a unique classical solution which is bounded in the strip $J\times\Rd$ for every bounded interval $J\subset [s,\infty)$, where by classical solution we mean a function $\uu\in C^{1,2}((s,\infty)\times\Rd;\Rm)\cap C([s,\infty)\times\Rd;\Rm)$ which solves \eqref{pioggia}.

We begin by considering the following proposition.

\begin{prop}
\label{prop:confronto_1}
Let Hypotheses $\ref{hyp-base}(i)$-$(iv)$ be satisfied. Then, for each $n\in\N$, $\f\in C_b(\R^d;\R^m)$ and $j\in\{1,\ldots,m\}$ it holds that $|(\G_n^{\mathcal N}(t,s)\f)_j|\leq (\G_n^{\mathcal N,P}(t,s)|\f|)_j$ in $B_n$ for every $t>s$.
\end{prop}

\begin{proof}
We fix $n\in\N$, $\f\in C_b(\Rd;\Rm)$ and set $\uu:=\G^{\mathcal N}_n(\cdot,s)\f$, $\bm v:=\G^{\mathcal N,P}_n(\cdot,s)|\f|$. For every $k=1,\ldots,m$ we denote by $G_{n,k}^{\mathcal N}(t,s)$ the positive evolution operator associated in $C(\overline{B_n})$ to the realization of the scalar operator
${\mathcal A}_{n,k}=\sum_{i,j=1}^d q_{ij}^kD_{ij}+\sum_{j=1}^db_{j}^kD_j+c_{kk}$,
 with homogeneous Neumann boundary conditions, and note that for every $T>s$ we can estimate $\|G^{\mathcal N}_{n,k}(t,s)\|_{L(C(\overline{B_n}))}\leq e^{c_{T,k,n}}$ for every $k=1,\ldots,m$ and $t\in[s,T]$, where $c_{k,T,n}:=\sup_{(t,x)\in (s,T)\times B_n}c_{kk}(t,x)$. 
 
 Let us fix $T>s$ and prove the statement for $t\in[s,T]$; the arbitrariness of $T$ will yield the assertion. By the variation-of-constants formula, we can write
\begin{align*}
u_k(t,x)
= & (G^{\mathcal N}_{n,k}(t,s)f_k)(x)+\int_s^t\bigg (G^{\mathcal N}_{n,k}(t,r)\sum_{j\neq k}c_{kj}(r,\cdot)u_j(r,\cdot)\bigg )(x)dr,\\[1mm]
v_k(t,x)
= & (G^{\mathcal N}_{n,k}(t,s)|f_k|)(x)+\int_s^t\bigg (G^{\mathcal N}_{n,k}(t,r)\sum_{j\neq k}c^P_{kj}(r,\cdot)v_j(r,\cdot)\bigg )(x)dr
\end{align*}
for every $(t,x)\in[s,T]\times\overline{B_n}$. Moreover, the positivity of the evolution operator $G^{\mathcal N}_{n,k}(t,s)$ implies that
\begin{align*}
|u_k(t,x)|\leq & (G^{\mathcal N}_{n,k}(t,s)|f_k|)(x)+\int_s^t\bigg (G^{\mathcal N}_{n,k}(t,r)\sum_{j\neq k}c^P_{kj}(r,\cdot)|u_j(r,\cdot)|\bigg )(x)dr
\end{align*}
for every $(t,x)\in[s,T]\times B_n$, so that
\begin{align*}
w_k(t,x)
\le & \int_s^t\bigg (G^{\mathcal N}_{n,k}(t,r)\sum_{j\neq k}c^P_{kj}(r,\cdot)w_j(r,\cdot)\bigg )(x)dr, \qquad\;\, (t,x)\in[s,T]\times B_n,
\end{align*}
where $\ww=|\uu|-\vv$ satisfies the inequality $\ww(s,\cdot)\le\bm 0$ in $B_n$.

To show that $\ww\le\bm 0$ on $[s,T]\times B_n$, we just need to prove that, if
\begin{align}
w_k(t,x)
\le & \int_{\sigma}^t\bigg (G^{\mathcal N}_{n,k}(t,r)\sum_{j\neq k}c^P_{kj}(r,\cdot)w_j(r,\cdot)\bigg )(x)dr
\label{quattro_1}
\end{align}
for every $(t,x)\in (\sigma,T)\times B_n$, $k=1,\ldots,m$ and $\bm w(\sigma,\cdot)\le\bm 0$ in $B_n$ for some $\sigma\in [s,T)$, then there exists $\delta>0$ such that $\bm w\le\bm 0$ in $[\sigma,\sigma+\delta]\times B_n$.
Indeed, once this property is proved, we can introduce the set
$E=\{t\in [s,T]:\ww(t,\cdot)\leq\bm 0{\textrm{ in }}B_n\}$, which contains $s$ and is an interval due to \eqref{quattro_1}.
Note that $E=[s,T]$. By contradiction, assume that $\tau=\sup E<T$. By continuity, it follows that $\ww(r,\cdot)\leq\bm 0$ in $B_n$  for every $r\in [s,\tau]$.
Since $\displaystyle \sum_{j\neq k}c^P_{kj}(r,\cdot)w_j(r,\cdot)\le 0$ in $B_n$ for every $r\in [s,\tau]$, from \eqref{quattro_1} we can
write
\begin{eqnarray*}
w_k(t,x)
\le \int_{\tau}^t\bigg (G^{\mathcal N}_{n,k}(t,r)\sum_{j\neq k}c^P_{kj}(r,\cdot)w_j(r,\cdot)\bigg )(x)dr, \qquad\;\, (t,x)\in[\tau,T]\times B_n
\end{eqnarray*}
and from this inequality we conclude that $\bm w\le\bm 0$ in a right neighborhood of $\tau$, which is a contradiction.

Let us prove \eqref{quattro_1} by contradiction. For this purpose, for every $h\in\N$ we denote by $J_h$ the set of all the indexes in $\{1,\ldots,m\}$ such that the condition $w_j\le 0$ is not satisfied in $[\sigma,\sigma+h^{-1}]\times\overline{B_n}$, where $h\in\N$ is such that $h^{-1}<T-\sigma$. Note that there exists at least an index $j_0$ such that $j_0\in J_h$ for infinite values of $h$ (let us say for every $h$ in an increasing sequence of natural numbers $(h_k)$). Finally, denote by $J$ the set of all $j\in\{1,\ldots,m\}$ such that $j\in J_{h_k}$ for every $k\in\N$.
For such values of $j$, we can determine
a decreasing sequence $(t_{k,j})\subset (\sigma,T]$ converging to $\sigma$ and $(x_{k,j})\subset\overline{B_n}$ such that $w_j(t_{k,j},x_{k,j})>0$ for every $k\in\N$. This implies that $h_j(r):=\sup_{({\sigma},r)\times B_n}w_j>0$ for every $r>{\sigma}$ and $j\in J$. Note that each function $h_j$ is continuous in $[{\sigma},\infty)$ since both $|\uu|$ and $\vv$ are continuous in $[s,T]\times \overline{B_n}$.
Let $\delta'>0$ be such that $w_i\le 0$ in $[\sigma,\sigma+\delta']\times B_n$ for every $j\in \{1,\ldots,m\}\setminus J$.
If $k\in J$, then we can estimate
\begin{align*}
w_{k}(t,x)
\leq  & \int_{\sigma}^t \bigg (G^{\mathcal N}_{n,k}(t,r)\sum_{k\neq j\in J}c_{kj}^P(r,\cdot)w_j(r,\cdot)\bigg )(x)dr \notag\\
\leq  & \int_{\sigma}^t \bigg (G^{\mathcal N}_{n,k}(t,r)\sum_{k\neq j\in J}c_{kj}^P(r,\cdot)h_j(r)\bigg )(x)dr \notag\\
\leq &  C\int_{\sigma}^t \sum_{j\in J}h_j(r)dr
\end{align*}
for every $(t,x)\in [\sigma,\sigma+\delta']\times\overline{B_n}$,
where we have used the fact that $c_{kj}^P\ge 0$ for every $j\neq k$, $h_j(r)>0$ for every $r\in({\sigma},T]$ and $j\in J$, and
\begin{align*}
C:=(e^{c_{T,k,n}}\vee 1)\sup_{{j,k\in J}\atop{j\neq k}}\|c_{kj}^P\|_{C([s,T]\times\overline{B_n})}<\infty.
\end{align*}
Hence, if we fix $t\in[\sigma,\sigma+\delta']$, then for every $\tau\in[\sigma,t]$ we can write
\begin{align*}
w_k(\tau,x)
\leq C\int_{\sigma}^{\tau} \sum_{j\in J}h_j(r)dr\leq C\int_\sigma^t \sum_{j\in J}h_j(r)dr.
\end{align*}
By taking the supremum of $(\tau,x)$ over $(\sigma,t)\times B_n$ and summing over all $k\in J$, we get
\begin{align*}
\sum_{k\in J}h_k(t)\leq Cm\int_{\sigma}^t \sum_{k\in J}h_k(r)dr, \qquad\;\, t\in[\sigma,\sigma+\delta'].
\end{align*}
Gronwall Lemma implies that $\sum_{k\in J}h_k(t)\leq 0$ in $[\sigma,\sigma+\delta']$, which contradicts the fact that $h_k(t)>0$ for every $t>{\sigma}$ and $k\in J$.
\end{proof}

\begin{prop}
\label{prop-2.6}
Under Hypotheses $\ref{hyp-base}$,
fix $s\in I$, $T>s$ and let $M:=M_{[s,T]}$ be the constant in \eqref{def_M_T}. Then, the following properties are satisfied.
\begin{enumerate}[\rm (i)]
\item
The unique solution $\uu^P$ to problem \eqref{cauchy_prob_pos} satisfies the estimate
\begin{align}
\max_{k=1,\ldots,d}|u^P_k(t,x)|\leq e^{M(t-s)} \max_{k=1,\ldots,m}\|f_k\|_\infty, \qquad\;\, (t,x)\in [s,T]\times \Rd.
\label{tre}
\end{align}
\item
There exists a unique classical solution $\uu$ to problem \eqref{pioggia}, which is bounded in each strip $[s,T]\times\Rd$. Moreover, if we denote by $\widetilde \uu^P$ the solution to \eqref{cauchy_prob_pos} with initial datum $|\f|$, then $|u_k(t,x)|\leq |\widetilde u_k^P(t,x)|$ for every $k=1,\ldots,m$ and $(t,x)\in[s,T]\times \Rd$. As a consequence,
\begin{align}
\label{stima_uni_gen}
\max_{k=1,\ldots,d}|u_k(t,x)|\leq e^{M(t-s)} \max_{k=1,\ldots,m}\|f_k\|_\infty, \qquad\;\, (t,x)\in [s,T]\times \Rd.
\end{align}
\end{enumerate}
\end{prop}

\begin{rmk}
{\rm Differently from \cite[Theorem 2.4]{AngLor20}, even if the off-diagonal entries of the matrix $C$ assume negative values, the constant $M$ in the above estimates does not depend on the lower bound of the off-diagonal entries of the matrix $C$, but only on the greatest upper bound of the sum of the elements on each row of $C^P$. This is a consequence of Proposition \ref{prop:confronto_1}, as we will see in the subsequent proof.}
\end{rmk}

\begin{proof}[Proof of Proposition $\ref{prop-2.6}$]
(i) Let us set $\vv(t,x):=\uu^P(t,x)-e^{M(t-s)}\max_{h=1,\ldots,m}\|f_h\|_\infty \1$ for each $(t,x)\in [s,T]\times \Rd$.
For every $k\in \{1,\ldots,m\}$,
$D_tv_k-(\A^P\vv)_k$ is non positive in $(s,T]\times \Rd$ and $\vv(s,\cdot)\leq \bm 0$ in $\Rd$. The maximum principle in \cite[Theorem 2.3]{AngLor20} (whose proof holds true also in our situation) implies that $u^P_k(t,x)\leq e^{M(t-s)} \max_{h=1,\ldots,m}\|f_h\|_\infty$ for every $(t,x)\in[s,T]\times \Rd$ and
$k\in\{1,\ldots,m\}$. Replacing $\uu^P$ with the function $-\uu^P$, we obtain the inequality $u^P_k(t,x)\geq -e^{M(t-s)} \max_{k=1,\ldots,m}\|f_k\|_\infty$ for $(t,x)\in[s,T]\times \Rd$ and $k\in\{1,\ldots,m\}$, and \eqref{tre} follows.

\medskip

(ii) We split the proof into two parts. In the first one we show that if there exists a classical solution which is bounded in each strip $[s,T]\times\Rd$, then it is unique. In the second part we construct a classical solution which satisfies \eqref{stima_uni_gen}.

{\it Uniqueness}. Let $\uu$ be a classical solution to \eqref{pioggia}
with $\f=\bm 0$, which is bounded in each strip $[s,T]\times\Rd$, and let us define the function $\ww=(w_1,\ldots,w_m)$ by setting
$w_k=u_k^2$ for $k=1,\ldots,m$.
Note that
\begin{align}
D_tw_k=2u_kD_tu_k={\mathcal A}_{k,0}w_k-2\langle Q_k\nabla u_k,\nabla u_k\rangle+2u_k(Cu)_k,\qquad\;\,k=1,\ldots,m,
\label{carboni}
\end{align}
in $(s,T]\times\Rd$, where ${\mathcal A}_{k,0}=\sum_{i,j=1}^dq_{ij}^kD^2_{ij}+\sum_{i=1}^d b_i^kD_i$.

Let us estimate the last term in the last side of \eqref{carboni}. For this purpose, we observe that
\begin{align*}
2u_k(t,x)\sum_{j=1}^m c_{kj}(t,x)u_j(t,x)\leq & 2c_{kk}(t,x)|u_k(t,x)|^2+2\sum_{j\neq k}|c_{kj}(t,x)||u_k(t,x)| |u_j(t,x)| \\
\leq & 2c_{kk}(t,x)|u_k(t,x)|^2+\sum_{j\neq k}|c_{kj}(t,x)|(|u_k(t,x)|^2+ |u_j(t,x)|^2) \\
= &  |u_k(t,x)|^2\bigg (c_{kk}(t,x)+\sum_{j\neq k}|c_{kj}(t,x)|\bigg ) \\
&+ \bigg (c_{kk}(t,x)|u_k(t,x)|^2+\sum_{j\neq k} |c_{kj}(t,x)||u_j(t,x)|^2 \bigg )\\
= & |u_k(t,x)|^2\sum_{j=1}^mc^P_{kj}(t,x) +\sum_{j=1}^mc^P_{kj}(t,x)|u_j(t,x)|^2 \\
\leq & Mw_k(t,x)+ \sum_{j=1}^m c^P_{kj}(t,x)w_j(t,x),
\end{align*}
which, combined with \eqref{carboni}, gives $D_t\ww-(\A^P+M{\rm Id})\ww \leq \bm 0$ in $(s,T]\times\Rd$. Again the maximum principle in \cite[Theorem 2.3]{AngLor20},
applied to the operator $D_t-(\A^P+M{\rm Id})$, implies that $\ww\leq \bm 0$, i.e., $\uu\equiv\bm 0$.

{\it Existence}.
We recall that for every $n\in\N$ the function $\G_n^{\mathcal N}(s,t)\f$ is the unique classical solution to the Cauchy problem with homogeneous Neumann boundary conditions
\begin{equation*}
\left\{
\begin{array}{lll}
D_t\uu(t,x)=(\A\uu)(t,x), & t\in(s,\infty), &x\in B_n, \vspace{1mm}\\
\displaystyle \frac{\partial \uu}{\partial \nu}(t,x)=\bm 0, & t\in(s,\infty), & x\in\partial B_n, \vspace{1mm} \\
\uu(s,x)=\f(x), && x \in B_n,
\end{array}
\right.
\end{equation*}
where $\nu$ is the outward unit normal to $B_n$. Arguing as in the proof of \cite[Theorem 2.4]{AngLor20} (see also Remark 2.5 therein and Proposition \ref{prop:app_stime_varie}$(i)$-$(ii)$ in the Appendix) it follows that, for every compact set $E\subset (s,\infty)\times \Rd$, the sequence $(\G_n^{\mathcal N}(s,t)\f)$ converges to a function $\uu$ in $C^{1,2}(E;\Rm)$, which is a classical solution to \eqref{pioggia}.

Proposition \ref{prop:confronto_1} shows that $|(\G_n^{\mathcal N}(\cdot,s)\f)_j|\leq (\G_n^{\mathcal N, P}(\cdot,s)|\f|)_j$ in $[s,\infty)\times\Rd$ for every $j\in\{1,\ldots,m\}$ and $n\in\N$. By letting $n$ tend to $\infty$ we deduce that $\G_n^{\mathcal N,P}(\cdot,s)|\f|$ converges to $\widetilde \uu^P$ in $C^{1,2}(E;\Rm)$ for every compact set $E\subset(s,\infty)\times \Rd$. From \eqref{tre} we get
\begin{align*}
\max_{k=1,\ldots,m}|u_k(t,x)|
= & \max_{k=1,\ldots,m}\lim_{n\rightarrow\infty} |((\G_n^{\mathcal N}(s,t)\f)(x))_k|
\leq \max_{k=1,\ldots,m}\lim_{n\rightarrow\infty}((\G_n^{\mathcal N, P}(\cdot,s)|\f|)(x))_k \\
=& \max_{k=1,\ldots,m}|\widetilde u_k^P(t,x)|\leq e^{M(t-s)}\max_{k=1,\ldots,m}\|f_k\|_\infty
\end{align*}
for every $(t,x)\in [s,T]\times \Rd$. This completes the proof.
\end{proof}

For further use, we consider the following results.

\begin{prop}
\label{pro:appr_sol_neumann}
In addition to Hypothesis $\ref{hyp-base}(i)$-$(ii)$, assume that the coefficients of the operator $\A$ belong to $C^{\alpha/2,3+\alpha}_{\rm loc}(I\times\Rd)$. Then, the following properties are satisfied for every $n\in\N$.
\begin{enumerate}[\rm (i)]
\item
for every $\f\in C(\overline{B_n};\R^m)$, the function $D^{\beta}_x\G^{\mathcal N}_n(\cdot,s)\f$ belongs to $C_{\rm loc}^{1+\alpha/2,2+\alpha}((s,T)\times B_n;\Rm)$ for every $\beta\in(\N\cup\{0\})^d$, with $|\beta|\le 3$, and $n\in\N$;
\item
the function $(t,x)\mapsto (t-s)^{(3-j)/2}|(D^3_x\G^{\mathcal N}_n(\cdot,s)\f)(x)|$ is continuous in $[s,T]\times B_n$ for every $n\in\N$, $\f\in C^j(\overline{B_n};\R^m)$ and $j=0,1,2,3$, where we extend it at $t=s$ in the trivial way;
\item
for each $s<t_2<t_1<T$, each pair of bounded sets $\Omega_1\Subset \Omega_2\Subset B_R$ and $n>R$ there exists a positive constant $C$, independent of $n$ and $\f\in C(\overline{B_n})$, such that $\|\G^{\mathcal N}_n(t,s)\f\|_{C^{1+\alpha/2,3+\alpha}((t_1,T)\times \Omega_1;\Rm)}\leq C\|\G^{\mathcal N}_n(t,s)\f\|_{L^\infty((t_2,T)\times \Omega_2;\Rm)}$. In particular, $\uu\in C^{1+\alpha/2,3+\alpha}_{\rm loc}((s,T)\times \Rd;\Rm)$ and $\G^{\mathcal N}_n(\cdot,s)\f$ converges to $\uu$ in $C^{1+\alpha/2,3+\alpha}(E)$ as $n$ tends to $\infty$ for every compact set $E\subset(s,T)\times \Rd$, where $\uu$ is the solution to the Cauchy problem \eqref{pioggia} provided by Proposition $\ref{prop-2.6}$.
\end{enumerate}
\end{prop}
\begin{proof}
(i) This property follows from \cite[Section II.4, Theorem 3.2]{Ei69}.

(ii) This property follows arguing as in \cite[Theorem 2.3]{BerLor05} and taking \cite[Chapter IV, Theorem 5.3]{LadSolUra68Lin} into account.

(iii) To simplify the notation, we set $\uu_n=G_n^{\mathcal N}(\cdot,s)\f$. 
Let us fix $s<t_2<t_3<t_1<T$ and three bounded open sets $\Omega_1\Subset\Omega_3\Subset\Omega_2$. Further, fix $R>0$ such that $\Omega_2\subset B_R$ and $n>R$. From property (i) we know that $D_{\ell}\uu_n\in C^{1+\alpha/2,2+\alpha}((s,T)\times B_n;\Rm)$ for every $\ell=1,\ldots,d$. Differentiating along the direction $e_\ell$, $\ell=1,\ldots,d$, the equation $D_t\uu_n=\A\uu_n$, we get
$D_tD_\ell\uu_n=\A D_\ell\uu_n+\g$ in $(s,T)\times B_n$, where
\begin{align*}
g_k=\sum_{i,j=1}^dD_{\ell}q_{ij}^kD_{ij}u_k+\sum_{i=1}^dD_{\ell}b_i^kD_iu_k+\sum_{j=1}^m D_{\ell}c_{kj}u_j,\qquad\;\,k=1,\ldots,m.
\end{align*}
By applying \cite[Theorem 7.2]{AngLor20} we infer that there exists a positive constant $C_1$ such that
\begin{align*}
\|D_{\ell}\uu_n\|_{C^{1+\alpha/2,2+\alpha}((t_1,T)\times \Omega_1;\Rm)}\leq C_1\big(\|D_{\ell}\uu_n\|_{L^\infty((t_3,T)\times\Omega_3;\Rm)}+\|\g\|_{C^{\alpha/2,\alpha}(t_3,T)\times \Omega_3;\Rm)}\big)
\end{align*}
and from the definition of $\g$, the assumptions on the coefficients of the operator $\A$ and applying once again \cite[Theorem 7.2]{AngLor20}, we infer that
\begin{align*}
\|D_{\ell}\uu_n\|_{L^\infty((t_3,T)\times\Omega_3;\Rm)}+\|\g\|_{C^{\alpha/2,\alpha}(t_3,T)\times \Omega_3;\Rm)}\leq C_2\|\uu_n\|_{L^\infty({t_2},T]\times \Omega_2;\Rm)}
\end{align*}
for some positive constant $C_2$, which is independent of $n$ as the constant $C_1$.
It follows that there exists a positive constant $C$ such that
\begin{align*}
\|\uu_n\|_{C^{1+\alpha/2,3+\alpha}((t_1,T)\times \Omega_1;\Rm)}\leq C_3 \|\uu_n\|_{L^\infty(({t_2},T]\times \Omega_2;\Rm)}
\end{align*}
for some constant $C_3>0$, independent of $n$. Since $\uu_n$ converges to $\uu$ in $C^{1+\alpha/2,2+\alpha}(E;\Rm)$ as $n$ tends to $\infty$ for each compact set $E\subset (s,T]\times \Rd$,
writing this estimate with $\uu_n$ being replaced by $\uu_n-\uu_m$, it follows that $(\uu_n)$ is a Cauchy sequence in $C^{1+\alpha/2,3+\alpha}((t_1,T)\times \Omega_1;\Rm)$ and the arbitrariness of $\Omega_1$ and $t_1$
implies that $\uu\in C_{\rm loc}^{1+\alpha/2,3+\alpha}((s,T)\times \Rd;\Rm)$ and $\uu_n$ converges to $\uu$ in $C^{1+\alpha/2,3+\alpha}(E;\Rm)$ as $n$ tends to $\infty$ for every compact set $E$ as above.
\end{proof}

In view of Proposition \ref{prop-2.6}, we can define an evolution operator $\G(t,s)$ on $C_b(\R^d;\R^m)$ by setting, for every $\f\in C_b(\R^d;\R^m)$, $\G(t,s)\f:=\uu(t,\cdot)$ if $t>s$ and $\G(s,s)\f:=\f$, where $\uu$ is the unique classical solution to \eqref{pioggia} with $\uu(s,\cdot)=\f$, which is bounded in each strip $[s,T]\times\Rd$. The uniqueness part of Proposition \ref{prop-2.6}  implies that $\G(t,r)\G(r,s)\f=\G(t,s)\f$ for every $t\geq r\geq s\in I$ and $\f\in C_b(\R^d;\R^m)$.
Similarly, we define the evolution operator $\G^P(t,s)$ on $C_b(\R^d;\R^m)$ associated with the parabolic Cauchy problem \eqref{cauchy_prob_pos}. In the following proposition, we prove some basic properties of the evolution operator $\G(t,s)$.
\begin{prop}
\label{prop:conv_loc_unif}
Fix $s\in I$ and a bounded sequence $(\f_n)\subset C_b(\Rd;\Rm)$. Then, the following properties are satisfied.
\begin{enumerate}[\rm (i)]
\item
If $\f_n$ pointwise converges to $\f\in C_b(\Rd;\Rm)$, then $\G(\cdot,s)\f_n$ converges to $\G(\cdot,s)\f$ in $C^{1,2}(E;\Rm)$ for every compact set $E\subset (s,\infty)\times \Rd$;
\item
If $\f_n$ locally uniformly converges to $\f\in C_b(\Rd;\Rm)$, then $\G(\cdot,s)\f_n$ converges to $\G(\cdot,s)\f$ locally uniformly in $[s,\infty)\times \Rd$.
\end{enumerate}
\end{prop}
\begin{proof}
The proof follows from \cite[Theorem 2.6]{AngLor20} recalling that
\begin{align*}
|(\G(t,s)\f_n)(x)-(\G(t,s)\f)(x)|\le (\G^P(t,s)|\f_n-\f|)(x), \qquad\;\, t>s, \;\, x\in\Rd.\hskip 3cm\qedhere
\end{align*}
\end{proof}

\begin{rmk}
\label{rmk-claudio}
{\rm We stress that all the results in this section can be extended to the autonomous case when the nonautonomous operator $\A$ is replaced with the autonomous operator $\A(\overline t)$, i.e., when $\overline t\in I$ is frozen.  In particular, for every $\overline t\in I$ we can introduce in the natural way the semigroups $\T_{\overline t}(\tau)$ and $\T_{\overline t}^P(\tau)$, associated, in $C_b(\Rd;\Rm)$, to the operators $\A(\overline t)$ and $\A^P(\overline t)$, respectively. Moreover,
\begin{align}
&\max\{\|\T_{\overline{t}}(\tau)\|_{L(C_b(\Rd;\Rm))},\|\T_{\overline{t}}^P(\tau)\|_{L(C_b(\Rd;\Rm))}\}
\le e^{M_{\overline{t}}\tau}\max_{k=1,\ldots,m}\|f_k\|_{\infty}
\label{stima_uni_gen_auto} 
\end{align}
for every $\f\in C_b(\Rd;\Rm)$ and $\tau>0$, where the constant $M_{\overline{t}}$ is defined by \eqref{def_M_T}, with $J=\{\overline t\}$.}
\end{rmk}

\section{Uniform estimates of the spatial derivatives}
\label{section:grad_est}
Throughout this section, we fix $T>s\in I$. 
We introduce the nonautonomous operator $\widetilde{\A}^P$, defined, componentwise on smooth functions $\f:\Rd\to\R^m$ by
\begin{align}
\label{operatore_A_tilde}
(\widetilde{\A}^P\f)_k={\rm Tr}(Q^kD^2f_k)+\langle
b^k,\nabla f_k\rangle+\widetilde C^P\f)_k
\end{align}
for every $t\in[s,T]$ and $k=1,\ldots,m$, where
$\widetilde c_{kj}^P(t,x):= c_{kj}^P(t,\cdot)+m^{-1}(1+|M_k(t,\cdot)|)$ for every $(t,x)\in [s,T]\times\Rd$ and $k,j=1,\ldots,m$. Here, $M_k$ is the function in \eqref{def_M_T}.

\begin{rmk}
{\rm We notice that
$\sum_{j=1}^m\widetilde c_{kj}^P(t,x)
=1+2(M_k(t,x))^+$ for every $(t,x)\in [s,T]\times\Rd$,
where $(M_k(t,x))^+$ denotes the positive part of $M_k(t,x)$ (see \eqref{def_M_T}). Hence, the potential matrix of the operator $\widetilde {\A}^P$ satisfies Hypothesis {\ref{hyp-base}$(v)$}: indeed,
\begin{align}
\label{stima_somma_tilde_c}
\max _{k=1,\ldots,m}\sup_{(t,x)\in[s,T]\times \Rd}\sum_{j=1}^m\widetilde c_{kj}^P(t,x)\leq 1+2M^+.
\end{align}
}
\end{rmk}

\begin{hyp}
\label{hyp:smoothness_coeff}
\begin{enumerate}[\rm (i)]
\item
The coefficients of the operator $\A$ belong to $C^{\alpha/2,3+\alpha}_{\rm loc}(I\times\Rd)$;
\item
there exists a function $r^k:[s,T]\times\Rd\to\R$ such that $\langle ({\rm J}_x\bm b^k)\xi,\xi\rangle\le r^k|\xi|^2$ on $[s,T]\times\Rd$ for every $\xi\in\Rd$;
\item
there exists a positive constant $C$ such that
\begin{align*}
\max\{|(Q^{k}(t,x)x)_i|, |{\rm Tr}(Q^{k}(t,x))|, \langle \bm b^k(t,x),x\rangle\}\leq C(1+|x|^2)\mu^k(t,x)
\end{align*}
for every $i=1,\ldots,d$, $k=1,\ldots,m$ and $(t,x)\in [s,T]\times \Rd$;
\item
there exist positive functions $b^{k,2}$, $b^{k,3}$ and a positive constant $L_k$ such that
\begin{align*}
|D^h_{x}b^k_j|\le b^{k,h},\qquad\;\,|\nabla_{x} q^k_{ij}|\le L_k\mu^k,\qquad\;\,
|D^h_{x}q^k_{ij}|\le L_k\mu^k
\end{align*}
on $(s,T]\times\Rd$ for every $i,j=1,\ldots,d$, $h=2,3$, and
$r^k+B_{k,2}b^{k,2}+B_{k,3}b^{k,3}\le {\mathscr M}_k\mu^k$ on $(s,T]\times\Rd$ for every $k=1,\ldots,m$
and some positive constants $B_{k,2}$, $B_{k,3}$ and ${\mathscr M}_k$;
\item
there exist positive functions $c^{k,h}$ such that $|D^j_xc_{k\ell}|\le c^{k,j}$
for every $k,\ell=1,\ldots,m$ and $h=1,\ldots,3$ and  $c^{k,h}\le \overline{C}(1+|M_k|)$ for every $h,k$ and some positive constant $\overline C$.
\end{enumerate}
\end{hyp}

The following one is the main result of this section.

\begin{thm}
\label{thm-3.4}
Let Hypotheses $\ref{hyp-base}$ and $\ref{hyp:smoothness_coeff}$ be fulfilled. Then,
\begin{enumerate}[\rm (i)]
\item for each $h,k=0,\ldots,3$, with $h\leq k$ and $T>s\in I$, there exists a constant $C=C_{h,k}(s,T)>0$ such that
\begin{align}
\label{stima_der_op_ev_totale}
\|\G(t,s)\f\|_{C_b^k(\Rd;\Rm)}\leq \frac{Ce^{\overline M (t-s)}}{(t-s)^{(k-h)/2}}\|\f\|_{C^h_b(\Rd;\Rm)}
\end{align}
for every $\f\in C^h_b(\Rd;\Rm)$ and $t\in(s,T]$, where $\overline M:=2^{-1}(1+M+2M^+)$;
\item 
for every $\mathcal T>0$, $\overline t\in [s,T]$ and $h,k=0,\ldots,3$, with $h\leq k$, there exists a constant $C=C_{h,k}(\overline t,\mathcal T)$, which is uniform with respect to $\overline t$ belonging to bounded sets, such that
\begin{align}
\label{stima_der_smgr_totale}
\|\T_{\overline t}(\tau)\f\|_{C_b^k(\Rd;\Rm)}\leq \frac{Ce^{\overline M \tau}}{\tau^{(k-h)/2}}\|\f\|_{C^h_b(\Rd;\Rm)}, 
\end{align}
for every $\f\in C^h_b(\Rd;\Rm)$ and $\tau\in (0,\mathcal T]$,
where in this case $\overline M:=2^{-1}(1+M_{\{\overline t\}}+2(M_{\{\overline t\}})^+)$.
\end{enumerate}
\end{thm}

\begin{proof}
(i) Let us first consider the case $h=0$ and $k=3$. We fix $\f\in C_b(\Rd;\Rm)$ and for every $n\in\N$ we denote by $\widetilde \uu_n$ be the unique classical solution to the Neumann-Cauchy problem associated to the operator $\A$ in $(s,T]\times B_n$, such that $\widetilde\uu_n(s,\cdot)=\vartheta_n\f$, where $\vartheta_n(x)=\varphi(n^{-1}|x|)$ for every $x\in\Rd$ and $\varphi\in C^\infty([0,\infty))$ is a nonincreasing function such that $\varphi\equiv1$ on $[0,1/2)$ and $\varphi\equiv 0$ on $[3/4,\infty)$. From Proposition \ref{pro:appr_sol_neumann}(i) it follows that the function $\widetilde \uu_n$ has spatial derivatives up to the third-order which belong to $C^{1+\alpha/2,2+\alpha}_{\rm loc}((s,T]\times B_n;\R^m)$ for every $n\in\N$. In the following, to lighten the notation we do not stress explicitly the dependence on $n$ and we simply write $\widetilde \uu$ and $\vartheta$.

Let us introduce the function $\vv=(v_1,\ldots,v_m)$, defined componentwise by
\begin{align*}
v_k(t,x)=&(\widetilde u_k(t,x))^2+\alpha(t-s)|\nabla_{x} \widetilde u_k(t,x)|^2+\alpha^2(t-s)^2(\vartheta(x))^2|D^2_{x}\widetilde u_k(t,x)|^2\\
&+\alpha^3(t-s)^3(\vartheta(x))^4|D^3_{x}\widetilde u_k(t,s)|^2\\
=&\!: w_k(t,x)+\alpha(t-s)w_k^1(t,x)+\alpha^2(t-s)^2(\vartheta(x))^2w_k^2(t,x)\\
&+\alpha^3(t-s)^3(\vartheta(x))^4w_k^3(t,x),
\end{align*}
for each $(t,x)\in(s,T]\times B_n$, where $\alpha\in(0,1]$ will be fixed later. 

Note that $\displaystyle \frac{\partial v_k}{\partial \nu}\leq 0$ on $(s,T]\times \partial B_n$ for every $k\in\{1,\ldots,m\}$. Indeed,
$\displaystyle\frac{\partial w_k}{\partial \nu}=2\widetilde u_k\frac{\partial \widetilde u_k}{\partial \nu}=0$ on $(s,T]\times \partial B_n$
and the normal derivative of $w_k^1$ is nonpositive since $B_n$ is convex (see \cite[Lemma 2.4]{BerFor04}). By Proposition \ref{pro:appr_sol_neumann}(ii), $\vv_n$ is continuous on $[s,\infty)\times B_n$.

In the following computations, we do not stress the dependence on $t$ and $x$ when it is not really needed, and we stress that the all the constants which appear depend neither on $n$ nor $x$, but depend on the interval $[s,T]$.
A long but straightforward computation reveals that
\begin{align*}
D_tv_k=&
\alpha w^1_k+2\vartheta^2\alpha^2(t-s)w^2_k+3\vartheta^4\alpha^3(t-s)^2w^3_k+{\mathscr I}_{\A,k}+{\mathscr I}_{Q,1,k}+{\mathscr I}_{Q,2,k}+{\mathscr I}_{Q,3,k}\\
&+{\mathscr I}_{{\bm b},1,k}
+{\mathscr I}_{{\bm b},2,k}+{\mathscr I}_{C,1,k}
+{\mathscr I}_{C,2,k}+{\mathscr I}_{Q,{\bm b},3}
\end{align*}
where
\begin{align*}
{\mathscr I}_{\A,k}=&2\widetilde u_k(\A\widetilde{\bm u})_k+2\alpha(t-s)\sum_{h=1}^dD_h\widetilde u_k(\A D_h\widetilde{\bm u})_k+2\vartheta^2\alpha^2(t-s)^2\sum_{h,p=1}^dD_{hp}\widetilde u_k(\A D_{hp}\widetilde{\bm u})_k\\
&+2\vartheta^4\alpha^3(t-s)^3\sum_{h,p,r=1}^dD_{hpr}\widetilde u_k(\A D_{hpr}\widetilde{\bm u})_k,\\
{\mathscr I}_{Q,1,k}=&2\alpha(t-s)\sum_{i,j,h=1}^dD_hq_{ij}^kD_h\widetilde u_kD_{ij}\widetilde u_k
+4\vartheta^2\alpha^2(t-s)^2\sum_{i,j,h,p=1}^dD_hq_{ij}^k D_{hp}\widetilde u_k D_{ijp}\widetilde u_k\\
&+6\vartheta^4\alpha^3(t-s)^3\sum_{i,j,h,p,r=1}^dD_rq_{ij}^kD_{hpr}\widetilde u_k D_{ijhp}\widetilde u_k,\\
{\mathscr I}_{Q,2,k}=&2\vartheta^2\alpha^2(t-s)^2\sum_{i,j,h,p=1}^dD_{hp}q_{ij}^kD_{ij}\widetilde u_kD_{hp}\widetilde u_k\\
&+6\vartheta^4\alpha^3(t-s)^3\sum_{i,j,h,p,r=1}^dD_{pr}q_{ij}^kD_{hpr}\widetilde u_kD_{ijh}\widetilde u_k,\\
{\mathscr I}_{{\bm b},1,k}=&2\alpha(t-s)\langle (J_x\bm b^k)\nabla_x \widetilde u_k,\nabla_x\widetilde u_k\rangle
+4\vartheta^2\alpha^2(t-s)^2\sum_{p=1}^d\langle (J_x\bm b^k)\nabla_x D_p\widetilde u_k,\nabla_x D_p\widetilde u_k\rangle\\
&+6\vartheta^4\alpha^3(t-s)^3\sum_{p,r=1}^d\langle (J_x\bm b^k)\nabla_x D_{pr}\widetilde u_k,\nabla_x D_{pr}\widetilde u_k\rangle,\\
{\mathscr I}_{{\bm b},2,k}=&2\vartheta^2\alpha^2(t-s)^2
\sum_{j,h,p=1}^dD_{hp}b_j^kD_j\widetilde u_kD_{hp}\widetilde u_k \\
&+6\vartheta^4\alpha^3(t-s)^3\sum_{j,h,p,r=1}^dD_{pr}b_j^k D_{jh}\widetilde u_k D_{hpr}\widetilde u_k,\\
{\mathscr I}_{C,1,k}=&2\alpha(t-s)\sum_{\ell=1}^m\sum_{h=1}^dD_hc_{k\ell} \widetilde u_{\ell} D_h\widetilde u_k
+4\vartheta^2\alpha^2(t-s)^2\sum_{\ell=1}^m\sum_{h,p=1}^dD_hc_{k\ell} D_p\widetilde u_{\ell} D_{hp}\widetilde u_k\\
&+6\vartheta^4\alpha^3(t-s)^3\sum_{\ell=1}^m\sum_{h,p,r=1}^dD_hc_{k\ell} D_{pr}\widetilde u_{\ell} D_{hpr}\widetilde u_k,\\
{\mathscr I}_{C,2,k}= & 2\vartheta^2\alpha^2(t-s)^2\sum_{\ell=1}^m\sum_{h,p=1}^d\!\!D_{hp}c_{k\ell} \widetilde u_{\ell} D_{hp}\widetilde u_k \\
& +6\vartheta^4\alpha^3(t-s)^3\sum_{\ell=1}^m\sum_{h,p,r=1}^d\!\!D_{hp}c_{k\ell}D_r\widetilde u_{\ell} D_{hpr}\widetilde u_k,\\
{\mathscr I}_{Q,{\bm b},C,3}=&2\vartheta^4\alpha^3(t-s)^3\bigg (\sum_{i,j,h,p,r=1}^dD_{hpr}q_{ij}^k D_{ij}\widetilde u_k D_{hpr}\widetilde u_k
+\sum_{j,h,p,r=1}^dD_{hpr}b_j^k D_j\widetilde u_k D_{hpr}\widetilde u_k\\
&\phantom{+2\vartheta^4\alpha^3(t-s)^3\bigg (\;\;\,}
+\sum_{\ell=1}^m\sum_{h,p,r=1}^dD_{hpr}c_{k\ell}\widetilde u_{\ell} D_{hpr}\widetilde u_k\bigg ).
\end{align*}

Since for every $\f\in C^1(B_n;\Rm)$ it holds that $f_k(\A\f)_k\leq (\A^P\g)_k+Mg_k-2\mu^k\|\nabla f_k\|^2$, where $g_k:=f_k^2$ for every $k=1,\ldots,m$ (see the proof of Proposition \ref{prop-2.6}(ii)), applying this formula to $\widetilde u_k$ and its derivatives,
and setting  $w_k^4=\displaystyle\sum_{i,h,p,r=1}^d|D_{ihpr}\widetilde u_k|^2$, it follows that
\begin{align*}
{\mathscr I}_{\A,k}
\le & g_k+(\A^p\bm w)_k+Mw_k-2\mu^kw_k^1
+\alpha(t-s)[(\A^P\bm w^1)_k+Mw_k^1-2\mu^kw^2_k]\\
&+\sum_{i=2}^3\vartheta^{2i-2}\alpha^i(t-s)^i[(\A^P\bm w^i)_k+Mw^i_k-2\mu^kw^{i+1}_k]\\
\le & (\A^P\bm v)_k+M v_k -2\mu^k\sum_{i=1}^4\alpha^{i-1}(t-s)^{i-1}w_k^i,
\end{align*}
where
\begin{align*}
g_k= & -2\vartheta\alpha^2(t-s)^2[w_k^2+2\vartheta^2\alpha(t-s)w_k^3]{\rm Tr}(Q^kD^2\vartheta) -4\alpha^2(t-s)^2\vartheta\langle Q^k\nabla_{x} w^2_k,\nabla\vartheta\rangle\\
&-2\vartheta\alpha^2(t-s)^2[w_k^2+2\vartheta^2\alpha(t-s)w_k^3]\langle b^k,\nabla \vartheta\rangle
-8\alpha^3(t-s)^3\vartheta^3\langle Q^k\nabla_{x} w^3_k,\nabla\vartheta\rangle.
\end{align*}
Taking Hypothesis \ref{hyp:smoothness_coeff}$(iii)$ into account, we get $\max\{|{\rm Tr}[Q^kD^2\vartheta]|, |Q^k\nabla \vartheta|, -\langle b^k, \nabla \vartheta\rangle\} \leq C\mu^k$ for some positive constant $C$, which does not depend on $n$. By applying the Young inequality $2ab\leq \varepsilon a^2+\varepsilon ^{-1}b^2$ for every $a,b,\varepsilon>0$, we can estimate
\begin{align*}
g_k\leq & C\alpha^2(t-s)^2\mu^k(\vartheta+\sqrt{d}\varepsilon^{-1})w_k^2+C\alpha^2(t-s)^2\vartheta^2\mu^k[\sqrt{d}\varepsilon+\alpha(t-s)(\vartheta+\sqrt{d}\varepsilon^{-1})]w_k^3\\
&+C\sqrt d\varepsilon \alpha^3(t-s)^3\vartheta^4\mu^k w_k^4
\end{align*}
for some positive constant $C$ and every $\varepsilon>0$. Hence, taking $\varepsilon=\sqrt{\alpha}$, we get
\begin{align*}
{\mathscr I}_{\A,k}\leq & (\A^P\bm v)_k+M v_k-2\mu^k[w^1_k+\alpha(t-s)w_k^2+\alpha^2(t-s)^2\vartheta^2w^3_k+\alpha^3(t-s)^3\vartheta^4w_k^4]\\
&+C\alpha^{\frac{3}{2}}(t-s)^2\mu^k(\sqrt{\alpha}\vartheta+\sqrt{d})w_k^2+C\alpha^{\frac{5}{2}}(t-s)^2\vartheta^2\mu^k[\sqrt{d}+(t-s)(\sqrt{\alpha}\vartheta+\sqrt{d})]w_k^3\\
&+C\sqrt d\alpha^{\frac{7}{2}}(t-s)^3\vartheta^4\mu^k w_k^4.
\end{align*}

Now, we observe that, since $\langle (J_x\bm b^k)\xi,\xi\rangle\le r^k|\xi|^2$ for every $\xi\in\Rd$, we can estimate
\begin{align*}
{\mathscr I}_{{\bm b},1,k}
\le &2\alpha(t-s)r^kw_k^1+4\alpha^2(t-s)^2\vartheta^2r^kw_k^2+6\alpha^3(t-s)^3r^k\vartheta^4w_k^3.
\end{align*}

Using Cauchy-Schwarz inequality and the assumption $\|\nabla_{x} q_{ij}^k\|\le L_k\mu^k$ for every $i,j=1,\ldots,d$, we can estimate
\begin{align*}
2\sum_{i,j,h=1}^dD_hq_{ij}^kD_hf_kD_{ij}f_k\le &2\sum_{i,j=1}^d\|\nabla_{x} q_{ij}\|\|\nabla f_k\|\|D_{ij}f_k\|\\
\le &2L_kd\mu^k\|\nabla f_k\|\|D^2_{x} f_k\|\le L_kd\mu^k\Big (\alpha^{-\frac{1}{2}}\|\nabla f_k\|^2+\alpha^{\frac{1}{2}}\|D^2 f_k\|^2\Big )
\end{align*}
for every $\f\in C^2(B_n;\R^m)$. Applying this estimate to $\widetilde\uu$ and its spatial derivatives, gives
\begin{align*}
{\mathscr I}_{Q,1,k}\le L_kd\mu^k\sum_{i=1}^3i\vartheta^{2i-2}\alpha^i(t-s)^i \Big(\alpha^{-\frac{1}{2}}w^i_k+\alpha^{\frac{1}{2}}w^{i+1}_k\Big ).
\end{align*}
Similarly, using the assumption $\|D^2_{x}q_{ij}^k\|\le L_k\mu^k$ and $\|D^2_{x}b_j\|\le b^{k,2}$ for every $i,j=1,\ldots,d$, we can estimate
\begin{align*}
{\mathscr I}_{Q,2,k}\le &L_kd\mu^k(
2\vartheta^2\alpha^2(t-s)^2w_k^2+6\vartheta^4\alpha^3(t-s)^3w_k^3)\\
{\mathscr I}_{{\bm b},2,k}\le & \sqrt{d}b^{k,2}
\left[\vartheta^2\alpha^2(t-s)^2\Big(\alpha^{-\frac12}w^1_k+\alpha^{\frac12}w^2_k\Big)+\vartheta^4\alpha^3(t-s)^3\Big(\alpha^{-\frac12}w^2_k+\alpha^{\frac12}w^3_k\Big)\right].
\end{align*}

Next, using the conditions $\|D_x^hc_{k\ell}\|\le c^{k,h}$ for every $k,\ell=1,\ldots,m$ and $h=1,2,3$, we can estimate
\begin{align*}
2\sum_{\ell=1}^m\sum_{h=1}^dD_hc_{k\ell} f_{\ell} D_hf_k
\le & 2\sqrt{m}c^{k,1}\|\f\|{\|\nabla f_k\|} 
\le \sqrt{m}c^{k,1}\left (\alpha^{-\frac{1}{2}}\|\f\|^2+\alpha^{\frac{1}{2}}\|J\f\|^2\right ),
\end{align*}
for every $\f\in C^1(B_n;\R^m)$.
Applying this estimate to $\widetilde\uu$ and its spatial derivatives we get
\begin{align*}
{\mathscr I}_{C,1,k}
\leq & \sqrt m c^{k,1}\sum_{i=1}^3i\vartheta^{2i-2}\alpha^{i}(t-s)^{i}\bigg(\alpha^{-\frac{1}{2}}\sum_{\ell=1}^m w_\ell^i+\alpha^{\frac{1}{2}}\sum_{\ell=1}^mw_\ell^{i+1}\bigg ).
\end{align*}
Similarly, we estimate (with the convention that $w_\ell^0=w_\ell$ for every $\ell=1,\ldots,m$)
\begin{align*} 
{\mathscr I}_{C,2,k}\le \sqrt m c^{k,2}
\sum_{i=1}^2(2i-1)\vartheta^{2i}\alpha^{i+1}(t-s)^{i+1}\bigg(\alpha^{-\frac{1}{2}}\sum_{\ell=1}^mw_\ell^{i-1}+\alpha^{\frac{1}{2}}\sum_{\ell=1}^mw_\ell^{i+1}\bigg).
\end{align*}
Finally, using the conditions $\|D^3_{x}q^k_{ij}\|\le L_k\mu^k$, $\|D^3_{x}b^k_j\|\le b^{k,3}$ and $\|D^3_{x}c_{k\ell}\|\le c^{k,3}$ for every $j=1,\ldots,d$ and $\ell=1,\ldots,m$, we get
\begin{align*}
{\mathscr I}_{Q,{\bm b},C,3}\le 
\vartheta^4\alpha^3(t-s)^3\bigg [&\sqrt{d}b^{k,3}\left (\alpha^{-\frac{1}{2}}w^1_k+\alpha^{\frac{1}{2}}w^3_k\right ) +L_kd\mu^k\left (\alpha^{-\frac{1}{2}}w_k^2+\alpha^{\frac{1}{2}}w^3_k\right )\\
&\;+\sqrt{m}c^{k,3}\bigg (\alpha^{-\frac{1}{2}}\sum_{\ell=1}^mw_{\ell}+
\alpha^{\frac{1}{2}}
\sum_{\ell=1}^mw_{\ell}^3\bigg )\bigg ].
\end{align*}

Combining everything together, we conclude that
\begin{align*}
D_tv_k\leq & ((\A^P+M{\rm Id})\bm v)_k+{\mathscr H}_1w_k^1+{\mathscr H}_2\alpha(t-s)w_k^2+{\mathscr H}_3\alpha^2(t-s)^2\vartheta^2 w_k^3\\
&+{\mathscr H}_4\alpha^3(t-s)^3\vartheta^4w_k^4+{\mathscr H}_5\sum_{j=1}^mw_j +{\mathscr H}_6\alpha(t-s)\sum_{j=1}^mw_j^1 +{\mathscr H}_7\alpha^2(t-s)^2\vartheta^2\sum_{j=1}^mw_j^2\\
&+{\mathscr H}_8\alpha^3(t-s)^3\vartheta^4\sum_{j=1}^mw_j^3,
\end{align*}
where
\begin{align*}
{\mathscr H}_1= & \alpha-2\mu^k+\sqrt{\alpha}L_kd\mu^k(t-s)+\alpha(t-s)[2r^k+\sqrt{\alpha}(t-s)\sqrt db^{k,2}+\alpha^{\frac{3}{2}}(t-s)^2\sqrt db^{k,3}], \\[1mm]
{\mathscr H}_2= & 2\alpha-2\mu^k+L_kd\sqrt{\alpha}\mu^k+(t-s)[C\sqrt{\alpha}(\sqrt{\alpha}+\sqrt d)
+2\sqrt{\alpha}L_kd]\mu^k\\
&+\alpha(t-s)\vartheta^2[4r^k+2L_kd\mu^k+\sqrt{\alpha}\sqrt db^{k,2}+\sqrt{\alpha}(t-s)(3\sqrt db^{k,2}+L_kd\mu^k)]\\[1mm]
{\mathscr H}_3= & [3\alpha+(2\sqrt{\alpha}L_kd+C\sqrt d\sqrt{\alpha}-2)\mu^k]+\sqrt{\alpha}(t-s)\mu^k[C\sqrt{\alpha}+C\sqrt{d}+3L_kd]\\
&+\alpha(t-s)\vartheta^2[6r^k+(6L_kd+\sqrt{\alpha} L_kd)\mu^k+3\sqrt{\alpha}\sqrt db^{k,2}+\sqrt{\alpha}\sqrt{d}b^{k,3}]\\[1mm]
{\mathscr H}_4 = & [(C\sqrt d+3L_kd)\sqrt{\alpha}-2]\mu^k, \\
{\mathscr H}_5 = & \sqrt{\alpha}(t-s)\sqrt{m}[c^{k,1}+\alpha(t-s)c^{k,2}+\alpha^2(t-s)^2c^{k,3}], \\
{\mathscr H}_6= & \sqrt{\alpha}\sqrt{m}[ c^{k,1}+2(t-s)c^{k,1}+3\alpha(t-s)^2c^{k,2}], \\
{\mathscr H}_7 = & \sqrt{\alpha}\sqrt{m}[2 c^{k,1}+c^{k,2}+3(t-s)c^{k,1}], \\
{\mathscr H}_8 = & \sqrt{\alpha}\sqrt{m}(3c^{k,1}+3c^{k,2}+c^{k,3}).
\end{align*}

Let us separately estimate these terms. From Hypothesis \ref{hyp:smoothness_coeff} we can take $\alpha$ small enough such that
$2r^k+\sqrt{\alpha}(t-s)\sqrt db^{k,2}+\alpha^{\frac{3}{2}}(t-s)^2\sqrt db^{k,3}\le 2{\mathscr M}_k\mu^k$, where $\mathscr M_k$ has been defined in Hypothesis \ref{hyp:smoothness_coeff}$(iv)$, so that
\begin{align*}
{\mathscr H}_1 \leq & \alpha-2\mu^k+\sqrt{\alpha}L_kd\mu^k(T-s)+2\alpha(T-s){\mathscr M}_k\mu^k.
\end{align*}
Hence, taking $\alpha$ smaller, if needed, we can make ${\mathscr H}_1$ negative on $(s,T]\times B_n$.
Arguing similarly, we can make ${\mathscr H}_j$ ($j=2,3,4$) negative choosing $\alpha$ sufficiently small. As far as the remaining coefficients
${\mathscr H}_j$ are concerned, we observe that we can make each of them smaller than $m^{-1}(1+|M_k|)$, choosing $\alpha$ sufficiently small. This implies that
$D_tv_k \leq ((\widetilde {\A}^P+M{\rm Id})\vv)_k$ on $(s,T]\times B_n$, where $\widetilde {\A}^P$ is the operator defined in \eqref{operatore_A_tilde}. 

Now, we consider the function $\widetilde \vv$, defined by $\widetilde \vv(t,x)=\vv(t,x)-e^{2\overline M(t-s)}\|\f\|_\infty^2\1$ for every $(t,x)\in[s,T]\times B_n$, and observe that
\begin{align*}
D_t\widetilde v_k
\leq & ((\widetilde{\A}^P+M Id)\vv)_k -2\overline Me^{2\overline M(t-s)}\|\f\|_\infty^2 \\
= & ((\widetilde {\A}^P+M Id)\widetilde \vv)_k+\bigg (\sum_{h=1}^m\widetilde c_{kh}^P-(1+2M^+)\bigg )e^{2\overline M(t-s)}\|\f\|_\infty^2.
\end{align*}
From \eqref{stima_somma_tilde_c} it turns out that $\widetilde \vv_n$ satisfies the following system of inequalities
\begin{align*}
\left\{
\begin{array}{ll}
D_t\widetilde \vv_n(t,x)-( \widetilde {\A}^P+M{\rm Id})\widetilde \vv_n(t,x)\leq \bm 0, & (t,x)\in(s,T]\times B_n, \vspace{1mm}\\
\displaystyle \frac{\partial \widetilde \vv_n}{\partial\nu}(t,x)\leq \bm 0, & (t,x)\in(s,T]\times \partial B_n, \vspace{1mm} \\
\widetilde \vv_n(s,x)\leq \bm 0, & x\in \overline B_n.
\end{array}
\right.
\end{align*}
By applying \cite[Theorem 3.15]{Pr67}, with $ {\A}$ being replaced by $ \widetilde {\A}^P+ M{\rm Id}$, we conclude that $\widetilde \vv_n\leq \bm 0$ on $(s,T]\times B_n$, i.e.,
\begin{align*}
&(u_{n,k}(t,x))^2+\alpha(t-s)|\nabla_{x} u_{n,k}(t,x)|^2+\alpha^2(t-s)^2(\vartheta_n(x))^2|D^2_{x}u_{n,k}(t,x)|^2\\
&+\alpha^3(t-s)^3(\vartheta_n(x))^4|D^3_{x}u_{n,k}(t,s)|^2\le e^{2\overline M(t-s)}\|\f\|_\infty^2
\end{align*}
for every $n\in\N$ and $k=1,\ldots,m$, with the constant $\alpha$ being independent of $n$. To conclude the proof, it suffices to let $n$ tend to $\infty$ taking Proposition \ref{pro:appr_sol_neumann}$(iii)$ into account.

The other cases can be deduced from similar arguments applied to the function $\vv^h=(v_1^h,\ldots,v_m^h)$ ($h=1,2,3$), defined by
\begin{align*}
v_k^h(t,x)=&\!: w_k(t,x)+\alpha(t-s)^{(1-h)^+}w_k^1(t,x)+\alpha^2(t-s)^{(2-h)^+}(\vartheta(x))^2w_k^2(t,x)\\
&\,+\alpha^3(t-s)^{(3-h)^+}(\vartheta(x))^4w_k^3(t,x)
\end{align*}
for $k=1,\ldots,m$ and $(t,x)\in [s,T]\times\Rd$, where $\ww, \ww^1,\ww^2$ and $\ww^3$ are defined as above and $\alpha\in(0,1]$ is to be properly fixed. 

{\vspace{2mm}}
(ii) Let us fix $\overline t\in[s,T]$ and $\mathcal T>0$. Arguing as in the proof of $(i)$ we get \eqref{stima_der_smgr_totale}. We notice that the constant $C(\overline t,\mathcal T)$ is uniform with respect to $\overline t\in[s,T]$ since the constant $C(s,T)$ in \eqref{stima_der_op_ev_totale} depends on the interval $[s,T]$.
\end{proof}

From Theorem \ref{thm-3.4} we can estimate the behaviour of some H\"older norms of $\G(t,s)\f$ and of $\T_{\overline t}(\tau)\f$, $\overline t\in I$, when $t-s$ and $\tau$ tend to $0$, respectively.

\begin{coro}
Under the assumptions of Theorem $\ref{thm-3.4}$, for every $\beta,\theta\in [0,3]$ with $\beta\le\theta$, there exists a positive constant $C=C(\beta,\theta,s,T)$ such that
\begin{equation}
\|\G(t,s)\f\|_{C^{\theta}_b(\Rd;\Rm)}\le C\frac{e^{\overline M(t-s)}}{(t-s)^{(\theta-\beta)/2}}\|\f\|_{C^{\beta}_b(\Rd;\Rm)},\quad\;\,\f\in C_b^\theta(\Rd;\Rm), \ t\in(s,T],
\label{stima-11}
\end{equation}
where $\overline M:=2^{-1}(1+M+2M^+)$. Further, for every $\overline t\in [s,T]$, $\beta,\theta\in [0,3]$, with $\beta\le\theta$, and $\omega>0$ there exists a positive constant $K=K(\overline t,\beta,\theta,\omega)$, uniform with respect to $\overline t\in[s,T]$,  such that
\begin{equation}
\|\T_{\overline t}(\tau)\f\|_{C^{\theta}_b(\Rd;\Rm)}\le K\frac{e^{(\omega+\overline {M})\tau}}{\tau^{(\theta-\beta)/2}}\|\f\|_{C^{\beta}_b(\Rd;\Rm)},\quad\;\,\f\in C_b^\theta(\Rd;\Rm), \ \tau>0,
\label{stima-12}
\end{equation}
where $\overline M:=2^{-1}(1+M_{\{\overline t\}}+2(M_{\{\overline t\}})^+)$.
\end{coro}

\begin{proof}
When at least one between $\theta$ and $\beta$ is not integer, the proof follows from an interpolation argument. We illustrate it in the particular case when $\theta=2$ and $\beta\in (1,2)$.
It is well-known that $C^{\beta}_b(\Rd)$ is the real interpolation space of order $(\beta/2,\infty)$ between $C_b(\Rd)$ and $C_b^2(\Rd)$. Clearly, this implies that
$C^{\beta}_b(\Rd;\Rm)=(C_b(\Rd;\Rm);C^{\beta}_b(\Rd;\Rm))_{\beta/2,\infty}$. Interpolating the estimates (we notice that $M\leq \overline M$)
\begin{align*}
&\|\G(t,s)\|_{L(C_b(\Rd;\Rm),C_b^{2k}(\Rd;\Rm))}\le C^k\frac{e^{\overline M(t-s)}}{(t-s)^k},
\qquad\;\,k=0,1,\end{align*}
we obtain estimate \eqref{stima-11} in this case.

Let us prove estimate \eqref{stima-12} with $\beta,\theta\in\{0,1,2,3\}$, $\beta\leq \theta$, the general case follows by interpolation. In particular we consider the case $\beta=0$ and $\theta=3$, since for the other integer values of $\beta$ and $\theta$ we get the assertion by analogous arguments. Let us fix $\omega>0$, and let $T\geq 1$ be such that $e^{\omega \tau}\tau^{-3/2}\geq1$ for every $\tau>T$. For $\tau\in (0,T]$, from \eqref{stima_der_smgr_totale} we get
\begin{align*}
\|\T_{\overline t}(\tau)\f\|_{C_b^{3}(\Rd;\Rm)}\leq \frac{C_{0,3}(\overline t, T)e^{\overline M}}{\tau^{3/2}}\|\f\|_{C_b(\Rd;\Rm)}, \qquad\;\, \f\in C_b(\Rd;\Rm),\;\, \tau\in (0,T].
\end{align*}
If $\tau>T$, then we use the semigroup property to split $\T_{\overline t}(\tau)=\T_{\overline t}(T)\T_{\overline t}(\tau-T)$ and so
\begin{align*}
\|\T_{\overline t}(\tau)\f\|_{C_b^{3}(\Rd;\Rm)}
\le & \|\T_{\overline t}(T)\|_{L(C_b(\Rd;\Rm);C_b^3(\Rd;\Rm))}\|\T_{\overline t}(\tau-T)\f\|_{C_b(\Rd;\Rm)} \\
\leq & \frac{C_{0,3}(\overline t, T)e^{\overline M}}{T^{3/2}}e^{\overline M(\tau-T)}\|\f\|_{C_b(\Rd;\Rm)} \\
\leq &  \frac{C_{0,3}(\overline t, T)e^{(\omega+\overline M)\tau}}{\tau^{3/2}}\|\f\|_{C_b(\Rd;\Rm)}.
\end{align*}
Estimate \eqref{stima-12} follows with $K(\overline t, 0,3,\omega)=C_{0,3}(\overline t, T)$.
\end{proof}

\section{Maximal Schauder regularity for problems \eqref{ellipt} and \eqref{parab}}
\label{sec:max_reg}
\subsection{Continuity results}

In this subsection we study the continuity of $\T_t(\tau)$ also with respect to $t\in I$.
\begin{thm}
\label{thm:continuita_totale}
Let Hypotheses $\ref{hyp-base}$ and $\ref{hyp:smoothness_coeff}$ be satisfied, and fix $\f\in  C_b(\Rd;\Rm)$. Then:
\begin{enumerate}[(i)]
\item 
the function $I\times (0,\infty)\times \Rd\ni (t,\tau,x)\mapsto (\T_t(\tau)\f)(x)$ is continuous;
\item 
for every bounded sequence $(\f_n)\subset C_b(\Rd;\Rm)$, which converges to $\f$ locally uniformly on $\Rd$, the function $(t,\tau,x)\mapsto (\T_t(\tau)\f_n)(x)$ converges to the function $(t,\tau,x)\mapsto (\T_t(\tau)\f)(x)$ locally uniformly in $I\times (0,\infty)\times \Rd$ as $n$ tends to $\infty$.
\end{enumerate}

\end{thm}
\begin{proof}
(i). For reader's convenience, we split the proof into two steps. In the former, we prove the statement for the semigroup $\T_t^{\mathcal N,R}(\tau)$ generated by the realization of $\A$ in $C_b(B_R;\Rm)$ with homogeneous Neumann boundary conditions on $\partial B_R$ for every $R>0$, with $\Rd$ replaced by $B_R$. In the latter we conclude.

{\em Step 1.} We fix $R>0$, $t\in I$ and prove that for every $\tau\in[0,\infty)$ it holds that
\begin{align}
\label{stima_neumann_smgr}
\|\T_t^{\mathcal N,R}(\tau)\f\|_{C_b(B_R;\Rm)}\leq e^{M_{\{t\}}^P\tau}\|\f\|_{C_b(B_R;\Rm)}.
\end{align}
For this purpose, we set $\ww:=(w_1,\ldots,w_m)$ with $w_j:=((\T_t^{\mathcal N,R}(\tau)\f)_j)^2-e^{-2M_{\{t\}}\tau}\|\f\|_{\infty}$ for $j=1,\ldots,m$. Arguing as in the uniqueness part of the proof of Proposition \ref{prop-2.6}, we deduce that function $\ww$ satisfies the system of inequalities
\begin{align*}
\left\{
\begin{array}{ll}
D_t\ww(\tau,x)-({\A}^P(t)+MId)\ww(\tau,x)\leq \bm 0, & (\tau,x)\in(0,\infty)\times B_R, \vspace{1mm}\\
\displaystyle \frac{\partial \ww}{\partial\nu}(\tau,x)= \bm 0, & (\tau,x)\in(0,\infty)\times \partial B_R, \vspace{1mm} \\
\ww(s,x)\leq \bm 0, & x\in \overline B_R.
\end{array}
\right.
\end{align*}
From \cite[Theorem 3.15]{Pr67} it follows that $\ww\le \bm 0$, which gives \eqref{stima_neumann_smgr}. Let $(\f_n)\subset C_b^{2}(\Rd;\Rm)$ converge to $\f$ locally uniformly on $\Rd$, let $\overline t\in I$ and let $J=\{t\in I:|t-\overline t|\leq 1\}$. Then, for every $t\in J$ we can estimate
\begin{align}
\|(\T_t^{\mathcal N,R}(\tau)\f)(x)-(\T_{\overline t}^{\mathcal N,R}(\tau)\f)(x)\|
\leq & \|(\T_t^{\mathcal N,R}(\tau)\f)(x)-(\T_t^{\mathcal N,R}(\tau)\f_n)(x)\|\notag \\
& + \|(\T_t^{\mathcal N,R}(\tau)\f_n)(x)-(\T_{\overline t}^{\mathcal N,R}(\tau)\f_n)(x)\| \notag \\
& + \|(\T_{\overline t}^{\mathcal N,R}(\tau)\f_n)(x)-(\T_{\overline t}^{\mathcal N,R}(\tau)\f)(x)\|
\label{step_1_neumann_1}
\end{align}
for every $(\tau,x)\in[0,\infty)\times B_R$. From \eqref{stima_neumann_smgr} we infer that
\begin{align}
& \|(\T_t^{\mathcal N,R}(\tau)\f)(x)-(\T_{t}^{\mathcal N,R}(\tau)\f_n)(x)\|
+ \|(\T_{\overline t}^{\mathcal N,R}(\tau)\f_n)(x)-(\T_{\overline t}^{\mathcal N,R}(\tau)\f)(x)\| \notag \\
\le & 2e^{M_J\tau}\|\f_n-\f\|_{C_b(B_R;\Rm)},
\label{step_1_neumann_2}
\end{align}
which vanishes as $n$ tends to $\infty$, locally uniformly with respect to $\tau$, since $\f_n$ converges to $\f$ as $n$ tends to $\infty$, locally uniformly on $\Rd$. Further, if we set $\vv_n=\T_t^{\mathcal N,R}(\cdot)\f_n-\T_{\overline t}^{\mathcal N,R}(\cdot)\f_n$ then we get
\begin{align*}
\left\{
\begin{array}{ll}
D_t\vv_n(\tau,x)=\A\vv_n(\tau,x)+(\A-\A(\overline t))\T_{\overline t}^{\mathcal N,R}(\tau)\f_n(x), & (\tau,x)\in(0,\infty)\times B_R, \vspace{1mm}\\
\displaystyle\frac{\partial\vv_n}{\partial\nu}(\tau,x)=\bm 0, & (\tau,x)\in(0,\infty)\times \partial B_R, \vspace{1mm} \\
\vv_n(0,x)=\bm 0, & x\in B_R.
\end{array}
\right.
\end{align*}
The variation-of-constants formula gives
\begin{align*}
\vv_n(\tau,x)=\int_0^\tau \T_{t}^{\mathcal N,R}(\tau-s)\big ((\A(t)-\A(\overline t))\T_{\overline t}^{\mathcal N,R}(s)\f_n\big )(x)ds, \quad (\tau,x)\in[0,\infty)\times B_R.
\end{align*}
Hence, for every $\mathcal T>0$ we get
\begin{align}
\sup_{(\tau,x)\in[0,\mathcal T]\times B_R}\|\vv_n(\tau,x)\|
\leq &\sup_{\tau\in[0,\mathcal T]}\int_0^\tau e^{M_{J}(\tau-s)}\|\A(t)-\A(\overline t)\|_{0,B_R}\| \|\T_{\overline t}^{\mathcal N,R}(s)\f_n\|_{C_b^2(B_R;\Rm)}ds \notag \\
\leq  & \widetilde C \|\f_n\|_{C^2_b(B_R;\Rm)}\|\A(t)-\A(\overline t)\|_{0,B_R},
\label{step_1_neumann_3}
\end{align}
where $\widetilde C$ is a positive constant, uniform with respect to $\tau$ and $t$ on bounded sets (see \cite[Section IV.2, Theorem 3.4]{Ei69}).

Let us fix $\varepsilon>0$, $\overline t\in I$ and let $J$ be as above. From \eqref{step_1_neumann_2} there exists $\overline n\in\N$ such that
\begin{align}
\label{step_1_neumann_4}
\|(\T_t^{\mathcal N,R}(\tau)\f)(x)-(\T_{t}^{\mathcal N,R}(\tau)\f_n)(x)\|
+ \|(\T_{\overline t}^{\mathcal N,R}(\tau)\f_n)(x)-(\T_{\overline t}^{\mathcal N,R}(\tau)\f)(x)\|\leq \frac{\varepsilon}{2},
\end{align}
for every $\N\ni n\geq \overline n$ and $t\in J$, locally uniformly with respect to $\tau\in[0,\infty)$ and uniformly with respect to $x\in B_R$. Since the coefficients of the operator $\A$ belong to $C^{\alpha/2,\alpha}_{\rm loc}(I\times \Rd)$, from \eqref{step_1_neumann_3} it follows that there exists $\delta>0$ such that
\begin{align}
\label{step_1_neumann_5}
\|(\T_t^{\mathcal N,R}(\tau)\f_{\overline n})(x)-(\T_{\overline t}^{\mathcal N,R}(\tau)\f_{\overline n})(x)\|
\leq \frac\varepsilon2, \quad {\rm if } \ |t-\overline t|\leq \delta,\ t\in J,
\end{align}
locally uniformly with respect to $\tau\in [0,\infty)$ and uniformly with respect to $x\in B_R$. Putting together \eqref{step_1_neumann_4} and \eqref{step_1_neumann_5}, from \eqref{step_1_neumann_1} it follows that for every $\varepsilon>0$ there exists $\delta>0$ such that
$\|(\T_t^{\mathcal N,R}(\tau)\f)(x)-(\T_{\overline t}^{\mathcal N,R}(\tau)\f)(x)\|\leq \varepsilon$ for every $t\in J$, such that $|t-\overline t|\leq \delta$,
locally uniformly with respect to $\tau\in[0,\infty)$, uniformly with respect to $x\in B_R$. Since $(\tau,x)\mapsto \T_{t}^{\mathcal N,R}(\tau)\f(x)$ is continuous in $(0,\infty)\times B_R$, pointwise with respect to $t$, the assertion follows.

{\it Step 2}. Here, we prove that, for every $\f\in C_b(\Rd;\Rm)$, $(\T_t^{\mathcal N,n}(\tau)\f)(x)$ converges to $ (\T_t(\tau)\f)(x)$ as $n$ tends to $\infty$ locally uniformly with respect to $(t,\tau,x)\in I\times (0,\infty)\times \Rd$. As a byproduct, from Step $1$ we obtain that the function $(t,\tau,x)\mapsto (\T_t(\tau)\f)(x)$ is continuous in $I\times (0,\infty)\times \Rd$. Let us fix $t\in I$ and observe that from the classical interior Schauder estimates it follows that,
for every compact set $E\in(0,\infty)\times \Rd$,
\begin{align}
\label{stima_interna_conv_totale}
\|\T_t^{\mathcal N,n}(\cdot)\f\|_{C^{1+\alpha/2,2+\alpha}(E;\Rm)}\leq C,
\end{align}
where $C$ is a positive constant which depends on $E$ and is locally uniform with respect to $t\in I$ (see Proposition \ref{prop:app_stime_varie}(iv)), but does not depend on $n$. Standard arguments (see e.g., \cite[Remark 2.5]{AngLor20}) imply that for every $t\in I$ the sequence $(\T_t^{\mathcal N,n}(\cdot)\f)$ converges to $\T_t(\cdot)\f$ in $C^{1,2}(E;\Rm)$, as $n$ tends to $\infty$, for every compact subset $E\subset(0,\infty)\times \Rd$. Hence, $(\T_t^{\mathcal N,n}(\tau)\f)(x)$ tends to $(\T_t(\tau)\f)(x)$ as $n$ tends to $\infty$, locally uniformly with respect to $(\tau,x)\in (0,\infty)\times \Rd$, pointwise with respect to $t\in I$. Let us assume by contradiction that such a convergence is not locally uniform with respect to $t$, i.e., there exist $[s,T]\subset I$, $[\varepsilon,\mathcal T]\subset (0,\infty)$ and $B_h\subset \Rd$ such that
\begin{align*}
\limsup_{n\rightarrow\infty}\sup_{t\in[s,T]}\|\T_t^{\mathcal N,n}(\cdot)\f-\T_t(\cdot)\f\|_{C_b([\varepsilon,\mathcal T]\times B_h)}\neq0.
\end{align*}
This implies that there exist $\varepsilon>0$, a sequence $(t_n)\subset [s,T]$ and an increasing sequence $(m_n)\subset \N$ such that
\begin{align}
\label{stima_contr_conv_totale}
\|\T_{t_n}^{\mathcal N,m_n}(\cdot)\f-\T_{t_n}(\cdot)\f\|_{C_b([\varepsilon,\mathcal T]\times B_h;\Rm)}\geq \varepsilon, \qquad\;\, n\in\N.
\end{align}
Let us notice that, up to a subsequence, $t_n$ converges to $\widehat t\in [s,T]$ as $n$ tends to $\infty$. Since the constant $C$ in \eqref{stima_interna_conv_totale} is uniform with respect to $t\in[s,T]$, it follows that for every compact set $E\subset (0,\infty)\times \Rd$ we can find a positive constant $C$ such that
\begin{align*}
\|\T_{t_n}^{\mathcal N,m_n}(\cdot)\f\|_{C^{1+\alpha/2,2+\alpha}(E;\Rm)}+\|\T_{t_n}(\cdot)\f\|_{C^{1+\alpha/2,2+\alpha}(E;\Rm)}\leq C, \qquad\;\, n\in\N.
\end{align*}
Hence, there exist $\vv,\ww\in C^{1+\alpha/2,2+\alpha}_{\rm loc}((0,\infty)\times \Rd;\Rm)$ such that, up to a subsequence, $\T_{t_n}^{\mathcal N,m_n}(\cdot)\f$ converges  to $\vv$ and $\T_{t_n}(\cdot)\f$ converges to $ \ww$ in $C^{1,2}(E;\Rm)$, as $n$ tends to $\infty$, and both $\vv$ and $\ww$ solve the equation $D_\tau\uu=\A(\widehat t)\uu$ on $(0,\infty)\times \Rd$.

Let us prove that $\vv\equiv\ww\equiv \T_{\widehat t}(\cdot)\f$ by means of a localization argument. Let us fix $M>0$, and let $\vartheta_M\in C_c^\infty(\Rd)$ satisfy the condition $\chi_{B_M}\leq \vartheta_M\leq \chi_{B_{M+1}}$. Let $\overline n\in\N$ be such that $m_n>M+1$ for every $n\geq \overline n$. For every  $n> \overline n$ the function $\vv_n:=\vartheta_M\T_{t_n}^{\mathcal N,m_n}(\cdot)\f$ satisfies the Neumann-Cauchy problem
\begin{align*}
\left\{
\begin{array}{ll}
D_\tau \vv_n(\tau,x)=\A(t_n)\vv_n(\tau,x)+\g_n(\tau,x), & (\tau,x)\in(0,\infty)\times B_{M+1}, \vspace{1mm}, \\
\displaystyle \frac{\partial \vv_n}{\partial\nu}(\tau,x)=\bm 0, & (\tau,x)\in (0,\infty)\times \partial B_{M+1}, \vspace{1mm} \\
\vv_n(0,x)=\vartheta_M(x)\f(x), & x\in B_{M+1},
\end{array}
\right.
\end{align*}
with
$g_{n,k}=-2\langle Q^k(t_n,\cdot)\nabla \vartheta_M,\nabla_{x} (\T_{t_n}^{\mathcal N,m_n}(\cdot)\f)_k\rangle -(\T_{t_n}^{\mathcal N,m_n}(\cdot)\f)_k{\mathcal A}_{k,0}(t_n)\vartheta_M$,
where ${\mathcal A}_{k,0}=\sum_{i,j=1}^dq_{ij}^kD^2_{ij}+\sum_{i=1}^d b_i^kD_i$, for every $k=1,\ldots,m$.
The variation-of-constants formula gives
\begin{align*}
\vv_n(\tau,x)=\T^{\mathcal N,M+1}_{t_n}(\tau)(\vartheta_M\f)(x)+\int_0^\tau\T_{t_n}^{\mathcal N,M+1}(\tau-s)(\g_n(s,\cdot))(x)ds
\end{align*}
for every $(\tau,x)\in(0,\infty)\times B_{M+1}$.
Since $\sqrt \tau\|\T^{\mathcal N,R}_{t}(\tau)\f\|_{C^1_b(B_{M+1};\Rm)}\leq K\|\f\|_{C_b(\Rd;\Rm)}$ for every $\tau\in(0,\mathcal T]$, some positive constant $K$ which is uniform with respect to $t\in[s,T]$ and $R>M+2$ (see Proposition \ref{prop:app_stime_varie}(iii) and \eqref{stima_neumann_smgr}), we get
\begin{align*}
\|(\T_{t_n}^{\mathcal N,m_n}(\tau)\f)(x)-\f(x)\|\leq \|(\T^{\mathcal N,M+1}_{t_n}(\tau)(\vartheta_M\f))(x)-\f(x)\|+K'\sqrt \tau\|\f\|_{C_b(\Rd;\Rm)}
\end{align*}
for every $(\tau,x)\in(0,\mathcal T]\times B_M$ and some positive constant $K'$, uniformly with respect to $(t,\tau)\in[s,T]\times (0,\mathcal T]$. Letting $n$ tend to $\infty$, we obtain that $\T_{t_n}^{\mathcal N,m_n}(\tau)\f(x)$ converges to $ \vv(\tau,x)$ and, from Step $1$, that $
\T^{\mathcal N,M+1}_{t_n}(\tau)(\vartheta_M\f)(x)$ converges to $\T^{\mathcal N,M+1}_{\widehat t}(\tau)(\vartheta_M\f)(x)$ for every $(\tau,x)\in(0,\mathcal T]\times B_M$. Hence,
\begin{align*}
\|\vv(\tau,x)-\f(x)\|\leq \|(\T^{\mathcal N,M+1}_{\widehat t}(\tau)(\vartheta_M\f))(x)-\f(x)\|+K'\sqrt \tau\|\f\|_{C_b(\Rd;\Rm)}
\end{align*}
for every $(\tau,x)\in(0,\mathcal T]\times B_M$,
and letting $\tau$ tends to $0^+$ we conclude that $\vv$ is continuous on $\{0\}\times B_M$. This means that $\vv$ is a classical solution to the Cauchy problem
\begin{align*}
\left\{
\begin{array}{ll}
D_t\uu(\tau,x)=\A(\widehat t)\uu(\tau,x), & (\tau,x)\in(0,\mathcal T]\times \Rd, \vspace{1mm} \\
\uu(0,x)=\f(x), & x\in\Rd,
\end{array}
\right.
\end{align*}
and so $\vv\equiv \T_{\widehat t}(\cdot)\f$. Similar arguments show that $\ww\equiv \T_{\widehat t}(\cdot)\f$. Letting $n$ tend to $\infty$ in \eqref{stima_contr_conv_totale} the assertion follows.

\medskip

(ii) We already know that, for every fixed $t\in I$, the sequence $\T_t(\cdot)\f_n$ converges to $\T_t(\cdot)\f$ locally uniformly on $[0,\infty)\times \Rd$ as $n$ tends to $\infty$ (see Proposition \ref{prop:conv_loc_unif}$(ii)$ and Remark \ref{rmk-claudio}). Let us assume by contradiction that there exist $[s,T]\subset I$, $[\varepsilon,\mathcal T]\subset (0,\infty)$ and $B_h\subset \Rd$ such that
\begin{align*}
\limsup_{n\rightarrow\infty}\sup_{t\in[s,T]}\|\T_t(\cdot)\f_n-\T_t(\cdot)\f\|_{C_b([\varepsilon,\mathcal T]\times B_h)}\neq0.
\end{align*}
This implies that there exist $\varepsilon>0$ , a sequence $(t_n)\subset [s,T]$ and an increasing and divergent sequence $(m_n)\subset \N$ such that
\begin{align}
\label{stima_contr_conv_totale_2}
\|\T_{t_n}(\cdot)\f_{m_n}-\T_{t_n}(\cdot)\f\|_{C_b([\varepsilon,\mathcal T]\times B_h;\Rm)}\geq \varepsilon, \qquad\;\, n\in\N.
\end{align}
Let us notice that, up to a subsequence, $t_n$ converges to $\widehat t\in [s,T]$ as $n$ tends to $\infty$. Arguing as in the proof of Step $2$ in (i), it follows that for every compact set $E\subset (0,\infty)\times \Rd$ we can find a positive constant $C$ such that
\begin{align*}
\|\T_{t_n}(\cdot)\f_{m_n}\|_{C^{1+\alpha/2,2+\alpha}(E;\Rm)}+\|\T_{t_n}(\cdot)\f\|_{C^{1+\alpha/2,2+\alpha}(E;\Rm)}\leq C, \qquad\;\, n\in\N.
\end{align*}
By applying a diagonal argument, we can find $\vv,\ww\in C^{1+\alpha/2,2+\alpha}_{\rm loc}((0,\infty)\times \Rd;\Rm)$ such that, up to a subsequence, $\T_{t_n}(\cdot)\f_{m_n}$ and $\T_{t_n}(\cdot)\f$ converge, respectively, to $\vv$ and to $ \ww$ in $C^{1,2}(E;\Rm)$, as $n$ tends to $\infty$, for every compact set $E\subset(0,\infty)\times\Rd$. Moreover, both $\vv$ and $\ww$ solve the equation $D_\tau{\bm z}=\A(\widehat t){\bm z}$ on $(0,\infty)\times \Rd$.

From property (i) we already know that $\ww\equiv\T_{\widehat t}(\cdot)\f$. Let us prove that $\vv\equiv \T_{\widehat t}(\cdot)\f$ by means of a localization argument. Let us fix $M>0$, and let $\vartheta_M\in C_c^\infty(\Rd)$ satisfy
the condition $\chi_{B_M}\leq \vartheta\leq \chi_{B_{M+1}}$. For every $n\in\N$ the function $\vv_n:=\vartheta_M\T_{t_n}(\cdot)\f_{m_n}$ satisfies the Neumann-Cauchy problem
\begin{align*}
\left\{
\begin{array}{ll}
D_\tau \vv_n(\tau,x)=\A(t_n)\vv_n(\tau,x)+\g_n(\tau,x), & (\tau,x)\in(0,\infty)\times B_{M+1}, \vspace{1mm}, \\
\displaystyle \frac{\partial \vv_n}{\partial\nu}(\tau,x)=\bm 0, & (\tau,x)\in (0,\infty)\times \partial B_{M+1}, \vspace{1mm} \\
\vv_n(0,x)=\vartheta_M(x)\f_{m_n}(x), & x\in B_{M+1},
\end{array}
\right.
\end{align*}
with the function $\g_n$ being defined as in Step 1, and with $\T_{t_n}^{\mathcal N, an m_n}(\cdot)\f$ being replaced by the function $\T_{t_n}(\cdot)\f_{m_n}$.
The variation-of-constants formula gives
\begin{align*}
\vv_n(\tau,x)=\T^{\mathcal N,M+1}_{t_n}(\tau)(\vartheta_M\f_{m_n})(x)+\int_0^\tau\T_{t_n}^{\mathcal N,M+1}(\tau-s)(\g_n(s,\cdot))(x)ds
\end{align*}
for each $(\tau,x)\in(0,\infty)\times B_{M+1}$.
Since $\sqrt \tau\|\T_{t}(\tau)\f_{m_n}\|_{C^1_b(B_{M+1};\Rm)}\leq K\sup_{n\in\N}\|\f_n\|_{\infty}$ for every $\tau\in(0,\mathcal T]$ and some positive constant $K$, which is uniform with respect to $t\in[s,T]$ (see Proposition \ref{prop:app_stime_varie}(iii)), it follows that
\begin{align*}
\|(\T_{t_n}(\tau)\f)(x)-\f(x)\|\leq \|(\T^{\mathcal N,M+1}_{t_n}(\tau)(\vartheta_M\f_{m_n}))(x)-\f(x)|+K'\sqrt \tau\sup_{n\in\N}\|\f_n\|_{C_b(\Rd;\Rm)},
\end{align*}
for every $(\tau,x)\in(0,\mathcal T]\times B_M$ and a constant $K'>0$, uniformly with respect to $n\in\N$. 

We claim that $\T^{\mathcal N,M+1}_{t_n}(\cdot)(\vartheta_M\f_{m_n})$ converges to $ \T^{\mathcal N,M+1}_{\widehat t}(\cdot)(\vartheta_M\f)$ as $n$ tends to $\infty$, uniformly in $(0,\mathcal T]\times B_{M+1}$. Once the claim is proved, letting $n$ tend $\infty$ and then $\tau$ tend to $0^+$, we get that $\uu$ is continuous on $\{0\}\times B_M$ for every $M>0$, and so it is a classical solution to
\begin{align*}
\left\{
\begin{array}{ll}
D_t\vv(\tau,x)=\A(\widehat t)\vv(\tau,x), & (\tau,x)\in(0,T]\times \Rd, \vspace{1mm} \\
\vv(0,x)=\f(x), & x\in\Rd.
\end{array}
\right.
\end{align*}
This means that $\vv\equiv \T_{\widehat t}(\cdot)\f$. Letting $n$ tend to $\infty$ in \eqref{stima_contr_conv_totale_2} we get a contradiction, and this yields the assertion.

It remains to prove the claim. From \eqref{stima_neumann_smgr} it follows that
\begin{align*}
& \|(\T^{\mathcal N,M+1}_{t_n}(\tau)(\vartheta_M\f_{m_n}))(x)- (\T^{\mathcal N,M+1}_{\widehat t}(\tau)(\vartheta_M\f))(x)\|\\
\leq & \|(\T^{\mathcal N,M+1}_{t_n}(\tau)(\vartheta_M\f_{m_n}))(x)- (\T^{\mathcal N,M+1}_{t_n}(\tau)(\vartheta_M\f))(x)\| \\
& +\|(\T^{\mathcal N,M+1}_{t_n}(\tau)(\vartheta_M\f))(x)- (\T^{\mathcal N,M+1}_{\widehat t}(\tau)(\vartheta_M\f))(x)\| \\
\leq & (e^{M_{[s,T]}^P\mathcal T}\vee 1)\|\vartheta_M \f_{m_n}-\vartheta_M\f\|_{C_b(B_{M+1};\Rm)} \\
& +\|(\T^{\mathcal N,M+1}_{t_n}(\tau)(\vartheta_M\f))(x)- (\T^{\mathcal N,M+1}_{\widehat t}(\tau)(\vartheta_M\f))(x)\| \\
=:& I_1(n)+I_2(n,\tau,x).
\end{align*}
Since $\f_n$ converges to $\f$ locally uniformly as $n$ tends to $\infty$, we infer that $I_1(n)$ converges to $0$. Further, from Step $1$ in the proof of property (i), it follows that $I_2(n,\tau,x)$ vanishes uniformly with respect to $(\tau,x)\in[0,\infty)\times B_{M+1}$. The claim is proved.
\end{proof}

\subsection{Schauder estimates}

In this subsection we prove optimal Schauder estimates for the solutions to the elliptic system
\begin{equation}
\lambda\uu(t,x)-\A(t)\uu(t,x)=\f(t,x), \qquad\;\, t\in [s,T],\;\, x\in \Rd
\label{ellipt}
\end{equation}
and to the Cauchy problem
\begin{equation}
\left\{
\begin{array}{lll}
D_t\uu(t,x)=\A(t)\uu(t,x)+\g(t,x), &t\in [s,T], &x\in\Rd,\\[1mm]
\uu(s,x)=\f(x), &&x\in\Rd,
\end{array}
\right.
\label{parab}
\end{equation}

At first we characterize the maximal domain of the realization of $\A(t)$ in $C_b(\Rd;\Rm)$ for every $t\in I$. To this aim, we fix $t\in I$ and for each $\lambda>M_{\{t\}}^P$ we consider the operator $\RR_t(\lambda):C_b(\Rd;\Rm)\to C_b(\Rd;\Rm)$ defined by
\begin{align}
\label{resolvent_operator_t}
\RR_t(\lambda)\f(x):=\int_0^\infty e^{-\lambda \tau}(\T_t(\tau)\f)(x)d\tau, \qquad\;\, \f\in C_b(\Rd;\Rm).
\end{align}
Thanks to \eqref{stima_uni_gen_auto},  $\RR_t(\lambda)$ is well defined and $\|\RR_t(\lambda)\|_{L(C_b(\Rd;\Rm))}\leq (\lambda-M_{\{t\}}^P)^{-1}$. Since the family $\{\RR_t(\lambda): \lambda>M_{\{t\}}^P\}$ satisfies the resolvent identity, there exists a closed operator $\B(t):D(\B(t))\subset C_b(\Rd;\Rm)\rightarrow C_b(\Rd;\Rm)$ such that $\RR_t(\lambda)=(\lambda-\B(t))^{-1}$ for each $\lambda >M_{\{t\}}^P$. Now, we prove that $\B(t)$ is a suitable realization of $\A(t)$ in $C_b(\Rd;\Rm)$, but to this aim we need a maximum principle for Problem \eqref{ellipt}.

\begin{prop}
\label{prop:ell_max_princ}
Fix $t\in I$, $\lambda>M_{\{t\}}^P\vee \lambda_{\{t\}}$ $($see Hypothesis $\ref{hyp-base}(iii))$. Suppose that $\uu\in C_b(\Rd;\Rm)\cap \bigcap_{p>1}W^{2,p}_{\rm loc}(\Rd;\Rm)$ is such that $\A^P(t)\uu\in
C(\Rd;\Rm)$ and satisfies the inequality $\lambda \uu-\A^P(t)\uu\leq \bm 0$. Then, $\uu\leq\bm
 0$ on $\Rd$.
\end{prop}
\begin{proof}
Let $t$ and $\lambda$ be as in the statement and introduce the function $\vv_n=\uu-n^{-1}\boldsymbol\varphi_{\{t\}}$, where $\boldsymbol \varphi_{\{t\}}$ is the function introduced in Hypothesis \ref{hyp-base}(iii). The choice of $\lambda$ and Hypothesis \ref{hyp-base}(iii) implies that
\begin{align}
\label{max_prin_ell_neg}
\lambda \vv_n(x)-(\A^P(t)\vv_n)(x)
<\bm 0,\qquad\;\,x\in\Rd.
\end{align}
Fix $n\in\N$. Since $\uu$ is bounded on $\Rd$ and $\bm\varphi_{\{t\}}(x)$ blows up, componentwise, as $|x|$ tends to $\infty$, for each $k\in \{1,\ldots,m\}$ the function $v_{n,k}$ attains its maximum value at some point $x_{n,k}\in\Rd$. Let $\overline k$ be such that
$v_{n,\overline k}(x_{n,\overline k})=\max_{j=1,\ldots,m}v_{n,j}(x_{n,j})$
and assume that $v_{n,\overline k}(x_{n,\overline k})>0$. Since ${\mathcal A}_{\overline k,0}(t)v_{n,\overline k}=(\A^P(t)\vv_n)_{\overline k}-(C^P(t,\cdot)\vv_n)_{\overline k}$ is a continuous function in $\Rd$, we can apply \cite[Theorem 3.1.10]{Lun95} and infer that ${\mathcal A}_{\overline k,0}(t)v_{\overline k}(x_{n,\overline k})\leq 0$. Thus, taking also into account that $(C^P(t,x_{n,\overline k})\vv(x_{n,\overline k}))_{\overline k}
\le \sum_{j=1}^mc^P_{\overline k,j}v_{n,\overline k}(x_{n,\overline k})$, we can estimate
$\lambda v_{n,\overline k}(x_{n,\overline k})-(\A^P(t)\vv_n)_{\overline k}(x_{n,\overline k})
\geq  (\lambda-M_{\{t\}}^P)v_{n,\overline k}(x_{n,\overline k})>0$, which contradicts \eqref{max_prin_ell_neg}. It follows that $\vv_n\le\bm 0$ on $\Rd$ for every $n\in\N$. Letting $n$ tend to $\infty$ we get $\uu\leq \bm 0$ on $\Rd$.
\end{proof}

\begin{prop}
\label{prop:carat_dom_At}
For every $t\in I$ it holds that
\begin{align*}
D(\B(t))=\bigg\{\f\in C_b(\Rd;\Rm)\cap\bigcap_{p>1}W^{2,p}_{\rm loc}(\Rd;\Rm):\A(t)\f\in C_b(\Rd;\Rm)\bigg\},
\end{align*}
and $\B(t)\f=\A(t)\f$ for every $\f\in D(\B(t))$. Further, $D(\B(t))$ is continuously embedded in $C_b^\theta(\Rd;\Rm)$ for every $\theta\in(0,2)$ and there exists a positive constant $C(\theta)$ such that
\begin{align*}
\|\f\|_{C_b^\theta(\Rd;\Rm)}\leq C(\theta)\|\f\|_\infty^{1-\theta/2}\|\f\|_{D(\B(t))}^{\theta/2}, \qquad\;\, \f\in D(\B(t)).
\end{align*}
\end{prop}
\begin{proof}
In the proof we separately provide the two inclusions, following the lines of \cite[Proposition 3.1]{DeLo11}.

{``$ \subset$''}.
Let $\uu\in D(\B(t))$ be such that $\uu=R(\lambda,\B(t))\f$ for some $\f\in C_b(\Rd;\Rm)$ and $\lambda>M_{\{t\}}^P$. For every $n\in\N$ we set
\begin{align*}
\uu_n(x):=\int_{1/n}^ne^{-\lambda \tau}(\T_t(\tau)\f)(x) d\tau , \qquad\;\, x\in\Rd.
\end{align*}
From \eqref{stima_uni_gen_auto} it follows that $\uu_n$ converges to $\uu$ in $C_b(\Rd;\Rm)$ as $n$ tends to $\infty$. Further, the spatial regularity of $\T_t(\tau)\f$ allows us to apply the operator $\A(t)$ to $\uu_n$, which gives
$\A(t)\uu_n
= \lambda \uu_n+e^{-\lambda n}\T_t(n)\f-e^{-\lambda/n}\T_t(1/n)\f$, in $\Rd$ for every $n\in\N$. This means that the sequence $(\A(t)\uu_n)$ is bounded in $C_b(\Rd;\Rm)$. By applying the classical $L^p$-interior estimates to each component $u_{n,k}$ of the function $\uu$, it follows that
\begin{align*}
\|u_{n,k}\|_{L^p(B_R)}
\leq & K\left(\|u_{n,k}\|_{L^p(B_{2R})}+\|\mathcal A_{k,0}(t)u_{n,k}\|_{L^p(B_{2R})}\right) \\
\leq & K\left(\|u_{n,k}\|_{L^p(B_{2R})}+\|(\A(t)\uu_{n})_k\|_{L^p(B_{2R};\Rm)}+\|(C\uu_n)_k\|_{L^p(B_{2R};\Rm)}\right)
\end{align*}
and the last side of the previous chain of inequalities is independent of $n$. Arguing as in the proof of \cite[Proposition 3.1]{DeLo11} we infer that $\uu\in W^{2,p}_{\rm loc}(\Rd;\Rm)$ for every $p\in[1,\infty)$, that $\A(t)\uu\in C_b(\Rd;\Rm)$ and that $\B(t)\uu=\A(t)\uu$ a.e. in $\Rd$.

``$ \supset$''. Let $\lambda\in \R$ be such that $2\lambda-M_{\{t\}}^P>M_{\{t\}}^P\vee \lambda_{\{t\}}$. We claim that if $\uu\in C_b(\Rd;\Rm)\cap\bigcap_{p>1}W^{2,p}_{\rm loc}(\Rd;\Rm)$ with $\A(t)\uu\in C_b(\Rd;\Rm)$ satisfies the equation $\lambda \uu-\A(t)\uu=\bm 0$ then $\uu=\bm 0$. Assume that the claim is true, and let $\f\in C_b(\Rd;\Rm)\cap \bigcap _{p>1}W^{2,p}_{\rm loc}(\Rd;\Rm)$ be such that $\A(t)\f\in C_b(\Rd;\Rm)$. Set $\boldsymbol \phi:=\lambda \f-\A(t)\f$ and $\g=R(\lambda, \B(t))\boldsymbol \phi\in D(\B(t))$. From the proof of the inclusion ``$\subset$'' it follows that $\lambda(\f-\g)-\A(t)(\f-\g)=\bm 0$, and the claim gives $\f=\g$, i.e., $\f\in D(\B(t))$.

To prove the claim, we fix $\uu$ as above and define the function $\ww$ by setting $w_k:=u_k^2$ for every $k\in\{1,\ldots,m\}$. It follows that
$2\lambda w_k
=2u_k(\A(t)\uu)_k
\leq \mathcal A_{k,0}(t)w_k+2u_k(C(t,\cdot)\uu)_k$
for $k=1,\ldots,m$.
Arguing as in the proof of Proposition \ref{prop-2.6}(ii) (uniqueness part), we infer that
$\mathcal A_{k,0}(t)w_k+2u_k(C(t,\cdot)\uu)_k
\leq (\A^P(t)\ww)_k+M_{\{t\}}^Pw_k$.
By combining these two last estimates we infer that
$(2\lambda-M_{\{t\}}^P)w_k-(\A^P(t)\ww)_k\leq 0$ for $k=1,\ldots,m$.
Since $2\lambda-M_{\{t\}}^P>M_{\{t\}}^P\vee \lambda_{\{t\}}$ and $\A^P(t)\ww\in C(\Rd;\Rm)$, from Proposition \ref{prop:ell_max_princ} it follows that $\ww\leq\bm 0$ on $\Rd$, which implies $\uu=\bm 0$ on $\Rd$. 

Let us sketch the proof of the last part of the statement. Fix $\bm 0\neq \f\in D(\B(t))$,  $\lambda> 1+2(M_{\{t\}}^P)^+$, and let $\bm \phi:=\lambda\f-\A(t)\f$. From \eqref{resolvent_operator_t} it follows that
$\f=\int_0^\infty e^{-\lambda \tau}\T_t(\tau){\bm \phi}d\tau$.
From estimate \eqref{stima-12}, we get
\begin{align*}
\|\f\|_{C^\theta_b(\Rd;\Rm)}
\leq & \frac{\Gamma(1-\theta/2)}{(\lambda-M_{\{t\}}^P)^{1-\theta/2}}\|\bm \phi\|_{C_b(\Rd;\Rm)} \\
\leq & \frac{\Gamma(1-\theta/2)}{(\lambda-M_{\{t\}}^P)^{1-\theta/2}}\big (\lambda \|\f\|_{C_b(\Rd;\Rm)}+\|\A(t)\f\|_{C_b(\Rd;\Rm)}\big ),
\end{align*}
and the assertion follows by minimizing with respect to $\lambda>1+2(M_{\{t\}}^P)^+$.
\end{proof}

To state the main result of this subsection. for $\theta\in (0,1)$ we denote by $C^{0,\theta}_b([s,T]\times\Rd;\R^m)$ the set of functions $\f\in C_b([s,T]\times\Rd;\R^m)$ such that $\sup_{t\in[s,T]}\|\f(t,\cdot)\|_{C^{\theta}_b(\Rd;\R^m)}<\infty$ and by
$C^{0,2+\theta}_b([s,T]\times\Rd;\R^m)$  the set of functions $\f\in C_b([s,T]\times\Rd;\R^m)$ which are twice continuously differentiable in $[s,T]\times\Rd$ with respect to the spatial variables, with $\sup_{t\in [s,T]}\|D_{ij}\f(t,\cdot)\|_{C^{\theta}_b(\Rd;\R^m)}<\infty$ for every $i,j=1,\ldots,d$.

\begin{thm}
\label{teo-Schau-est}
Fix $T>s\in I $. The following properties are satisfied.
\begin{enumerate}[\rm (i)]
\item
Let {$\f\in C_b^{0,\theta}([s,T]\times \Rd;\Rm)$} for some $\theta\in(0,1)$. Then, for every $\lambda>M_{[s,T]}^P$, the elliptic equation \eqref{ellipt} admits a unique solution $\uu\in C_b^{0,2+\theta}([s,T]\times \Rd;\Rm)$ and there exists a positive constant $C=C(\lambda,\theta)$ such that
\begin{equation}
\sup_{t\in[s,T]}\|\uu(t,\cdot)\|_{C^{2+\theta}_b(\Rd;\Rm)}\le{C\sup_{t\in[s,T]}\|\f(t,\cdot)\|_{C_b^\theta(\Rd;\Rm)}}.
\label{stima-112}
\end{equation}
\item Let $\f\in C^{2+\theta}_b(\Rd;\Rm)$ and $\g\in C^{0,\theta}_b([s,T]\times\Rd;\Rm)$ for some $\theta\in (0,1)$. Then, the Cauchy problem \eqref{parab} admits a unique classical solution $\uu\in C^{0,2+\theta}_b([s,T]\times\Rd;\Rm)$ and there exists a positive constant $C$ such that
\begin{equation*}
\sup_{t\in [s,T]}\|\uu(t,\cdot)\|_{C^{2+\theta}_b(\Rd;\Rm)}\le C\Big (\|\f\|_{C^{2+\theta}_b(\Rd;\Rm)}+\sup_{t\in [s,T]}\|\g(t,\cdot)\|_{C^{\theta}_b(\Rd;\Rm)}\Big ).
\end{equation*}
{Finally, $\uu\in C_b^{\theta/2}([s,T];C^2(\overline{B_R};\Rm))$ for every $R>0$.}
\end{enumerate}
\end{thm}
\begin{proof}
(i) {\it Uniqueness}. 
Fix $\lambda\in\R$ such that $2\lambda>(\lambda_{[s,T]}\vee M_{[s,T]})+M_{[s,T]}$ and 
let $\uu$ solve the equation $\lambda \uu(t,\cdot)-\A(t)\uu(t,\cdot)=\bm 0$ with $\uu(t,\cdot)\in D(\B(t))$ for every $t\in[s,T]$.  
Arguing as in the second part of the proof of Proposition \ref{prop:carat_dom_At} we infer that $\uu=\bm 0$. This shows that, for every $\f:[s,T]\times \Rd\rightarrow \Rm$ with $\f(t,\cdot)\in C_b(\Rd;\Rm)$, for every $\lambda\in\R$, such that $2\lambda>(\lambda_{[s,T]}\vee M_{[s,T]})+M_{[s,T]}$, and for every $t\in[s,T]$, there exists a unique solution $\vv$ to the equation $\lambda \vv(t,\cdot)-\A(t)\vv(t,\cdot)=\f(t,\cdot)$ given by $\vv(t,\cdot)=R(\lambda,\B(t))\f(t,\cdot)$ for $t\in[s,T]$. Based on this result, we extend the previous property to any $\lambda>M_{[s,T]}^P$. So, let $\uu$ be such that $\uu(t,\cdot)\in D(\B(t))$ and satisfy the equation $\lambda\uu(t,\cdot)-\A(t)\uu(t,\cdot)=\bm0$ for $t\in[s,T]$. Let $K>0$ be such that $\overline \lambda:=\lambda+K$ satisfies $2\overline \lambda>(\lambda_{[s,T]}\vee M_{[s,T]}^P)+M_{[s,T]}^P$. Clearly $\uu$ solves the equation $\overline \lambda\uu(t,\cdot)-\A(t)\uu(t,\cdot)=K\uu(t,\cdot)$ for every $t\in[s,T]$, and so $\uu(t,\cdot)=R(\overline \lambda,\B(t))(K\uu(t,\cdot))$ for every $t\in[s,T]$. We fix $t\in[s,T]$ and consider the operator $\Gamma_{t} :\left(C_b(\Rd;\Rm),\|\cdot\|_{\widetilde{\infty}}\right)\rightarrow \left(C_b(\Rd;\Rm),\|\cdot\|_{\widetilde{\infty}}\right)$ defined by $\Gamma_{t}(\g):=R(\overline \lambda,\B(t))(K\g)$, where $\|\f\|_{\widetilde\infty}$ denotes the maximum over $k=1,\ldots,m$ of the sup-norm of the functions $f_k$.
We claim that $\Gamma_{t}$ is a contraction. If the claim is true, then $\Gamma_t$ admits as unique fixed point the null function, and since $\uu(t,\cdot)$ is a fixed point for $\Gamma_t$, it follows that $\uu(t,\cdot)\equiv\bm 0$. The arbitrariness of $t$ gives $\uu\equiv\bm0$ in $[s,T]\times \Rd$. Let us prove the claim. From the definition of $R(\overline \lambda,\B(t))$, \eqref{stima_uni_gen_auto} and recalling that $\overline \lambda=\lambda+K$, for every $\g\in C_b(\Rd;\Rm)$ we get
\begin{align*}
\|\Gamma_t(\g)\|_{\widetilde \infty}
\leq K\|\g\|_{\widetilde \infty} \int_0^\infty e^{(-\overline \lambda+M_{[s,T]}^P)\tau}d\tau
= \frac{K}{\overline\lambda-M_{[s,T]}^P}\|\g\|_{\widetilde \infty}
= \frac{\overline \lambda-\lambda}{\overline\lambda-M_{[s,T]}^P}\|\g\|_{\widetilde \infty}.
\end{align*}
Since $\lambda>M_{[s,T]}^P$ it follows that
$({\overline \lambda-\lambda)}(\overline\lambda-M_{[s,T]}^P)^{-1}<1$. The claim is so proved and it thus follows that $\uu\equiv\bm 0$ in $[s,T]\times \Rd$.

\vspace{2mm}
{\it Existence}. Arguing as in \cite[Theorem 1]{Lun98} (see also \cite[Theorem 3.6]{BerLor05}), and taking advantage of estimate \eqref{stima-12} (which is uniform with respect to $t\in[s,T]$ if we replace $M_{\{t\}}^P$ with $M_{[s,T]}^P$), it follows that for every $\lambda>M_{[s,T]}^P$ and $t\in [s,T]$ there exists a unique solution $\vv(t)\in C_b^{2+\theta}(\Rd;\Rm)$ to the equation $\lambda\vv(t)-\A(t)\vv(t)=\f(t,\cdot)$ such that $\|\vv(t)\|_{C_b^{2+\theta}(\Rd;\Rm)}\leq C\|\f(t,\cdot)\|_{C_b^\theta(\Rd;\Rm)}$ where $C$ is a positive constant, which depends on $\lambda$ and $\theta$, and can be chosen uniform with respect to $t\in[s,T]$. If we set $\uu(t,x)=\vv(t)(x)$ for every $t\in[s,T]$ and $x\in\Rd$, then it follows that $\uu$ satisfies \eqref{stima-112}.

Let us prove that $\uu\in C_b^{0,2+\theta}([s,T]\times \Rd;\Rm)$. For every $t,\overline t\in[s,T]$ we can write
\begin{align*}
\uu(t,x)-\uu(\overline t,x)
= & \int_0^\infty e^{-\lambda \tau}(\T_t(\tau)\f(t,\cdot)(x)-\T_{\overline t}(\tau)\f(\overline t,\cdot)(x))d\tau, \qquad\;\, x\in \Rd.
\end{align*}
Let us consider the function under the integral sign and split
\begin{align*}
\T_t(\tau)\f(t,\cdot)(x)-\T_{\overline t}(\tau)\f(\overline t,\cdot)(x)
= & \left(\T_t(\tau)\f(t,\cdot)(x)-\T_{t}(\tau)\f(\overline t,\cdot)(x) \right)\\
& +\left(\T_t(\tau)\f(\overline t,\cdot)(x)-\T_{\overline t}(\tau)\f(\overline t,\cdot)(x) \right)\\
=: & I_1(t,\tau,x)+I_2(t,\tau,x).
\end{align*}
We estimate $\|I_1(t,\tau,x)\|
\leq \sup_{\sigma\in[s,T]} \|\T_\sigma(\tau)\f(t,\cdot)(x)-\T_{\sigma}(\tau)\f(\overline t,\cdot)(x) \|$
for every $(t,\tau,x)\in [s,T]\times (0,\infty)\times \Rd$. Next, we fix a compact set $E\subset (0,\infty)\times \Rd$, a sequence $(t_n)\subset [s,T]$ converging to $\overline t$, and notice that $\f(t_n,\cdot)$ converges to $\f(\overline t,\cdot)$ locally uniformly on $\Rd$, as $n$ tends to $\infty$. From Theorem \ref{thm:continuita_totale}(ii) it follows that
\begin{align*}
\lim_{n\rightarrow \infty}\|I_1(t_n,\cdot,\cdot)\|_{C_b(E;\Rm)}
\leq & \lim_{n\rightarrow\infty}\sup_{\sigma\in[s,T]}\|\T_{\sigma}(\cdot)\f(t_n,\cdot)-\T_{\sigma}(\cdot)\f(\overline t,\cdot) \|_{C_b(E;\Rm)}=0.
\end{align*}
The arbitrariness of the sequence $(t_n)$ implies that
$\lim_{t\rightarrow \overline t}\|I_1(t,\cdot,\cdot)\|_{C_b(E;\Rm)}=0$.

As far as $I_2$ is concerned, from Theorem \ref{thm:continuita_totale}(i) we infer that for every compact set $E\subset (0,\infty)\times \Rd$ it holds that
$\lim_{t\rightarrow \overline t}\|I_2(t,\cdot,\cdot)\|_{C_b(E;\Rm)}=0$.
By combining the previous two estimates we conclude that
\begin{align}
\label{stima_cont_ell_integr}
\lim_{t\to \overline t}\|\T_t(\cdot)\f(t,\cdot)-\T_{\overline t}(\cdot)\f(\overline t,\cdot)\|_{C_b(E;\Rm)}=0
\end{align}
for every compact set $E\subset (0,\infty)\times \Rd$. Let us fix $R>0$, then
\begin{align}
\|\uu(t,\cdot)-\uu(\overline t,\cdot)\|_{C_b({B_R};\Rm)}
\leq \int_0^\infty e^{-\lambda \tau}\|\T_t(\tau)\f(t,\cdot)-\T_{\overline t}(\tau)\f(\overline t,\cdot)\|_{C_b({B_R};\Rm)}d\tau
\label{SFV4S}
\end{align}
and from \eqref{stima_cont_ell_integr} we infer that the function under the integral sign vanishes as $t$ tends to $\overline t$, pointwise with respect to $\tau\in(0,\infty)$. Further, from \eqref{stima_uni_gen_auto} we infer that
\begin{align*}
\|\T_t(\tau)\f(t,\cdot)-\T_{\overline t}(\tau)\f(\overline t,\cdot)\|_{C_b({B_R};\Rm)}
\leq 2\sqrt m e^{M_{[s,T]}^P\tau}\sup_{\sigma\in[s,T]}\|\f(\sigma,\cdot)\|_\infty,
\end{align*}
for every $\tau\in(0,\infty)$. Since $\lambda>M_{[s,T]}^P$ the dominated convergence theorem implies that
the right-hand side of \eqref{SFV4S} vanishes as $t$ tends to $\overline t$, which shows that $\uu(t,\cdot)$ converges to $\uu(\overline t,\cdot)$ in $C_b({B_R};\Rm)$ as $t$ tends to $\overline t$, for every $R>0$. Finally, by interpolating between $C_b({B_R};\Rm)$ and $C_b^{2+\theta}({B_R};\Rm)$ we get
\begin{align}
\|\uu(t,\cdot)-\uu(\overline t,\cdot)\|_{C_b^2({B_R};\Rm)}
\!\leq &C\|\uu(t,\cdot)-\uu(\overline t,\cdot)\|_{C_b({B_R};\Rm)}^{\frac{\theta}{2+\theta}} \|\uu(t,\cdot)-\uu(\overline t,\cdot)\|_{C_b^{2+\theta}({B_R};\Rm)}^{\frac{2}{2+\theta}} \notag \\
\leq & \widetilde C\bigg (\sup_{\sigma\in[s,T]}\|\f(\sigma,\cdot)\|_{C_b^\theta(\Rd;\Rm)}\bigg )^{\frac{2}{2+\theta}}\|\uu(t,\cdot)-\uu(\overline t,\cdot)\|_{C_b({B_R};\Rm)}^{\frac{\theta}{2+\theta}},
\label{stima_cont_der_seconde_ell}
\end{align}
and the last term of \eqref{stima_cont_der_seconde_ell} tends to $0$ as $t$ approaches $\overline t$. Hence, the second-order spatial derivatives of $\uu$ are bounded and continuous on $[s,T]\times\Rd$. This completes the proof.

\vspace{2mm}
(ii) The existence and uniqueness part follows arguing as in \cite[Theorem 2]{Lun98} (see also \cite[Theorem 3.7]{BerLor05}). Since $\A(t)\uu(t,\cdot)\in C^\theta_b({B_R};\Rm)$ for every $t\in[s,T]$, by difference we infer that $t\mapsto D_t\uu(t,\cdot)\in C([s,T];C^\theta_b({B_R};\Rm))$ for every $R>0$ and
\begin{align*}
\|D_t\uu(t,\cdot)\|_{C^\theta_b({B_R};\R^m)}
\leq & \bigg (\sup_{t\in[s,T]}\|\A(t)\uu(t,\cdot)\|_{C_b^\theta(B_R;\Rm)}+\sup_{t\in[s,T]}\|\g(t,\cdot)\|_{C_b^{\theta}(B_R;\Rm)}\bigg ) \\
\leq & C\!\!\!\sup_{t\in[s,T]}\|\A(t)\|_{\theta,B_R}\Big (\|\f\|_{C^{2+\theta}_b(\Rd;\Rm)}+\sup_{t\in [s,T]}\|\g(t,\cdot)\|_{C^{\theta}_b(\Rd;\Rm)}\Big ) \\
& +\sup_{t\in [s,T]}\|\g(t,\cdot)\|_{C^{\theta}_b(\Rd;\Rm)}
\end{align*}
for every $t\in [s,T]$ and some positive constant $C$, independent of $t$.
Thus, we get
\begin{align*}
\|\uu(t,\cdot)-\uu(\tau,\cdot)\|_{C^2_b({B_R};\Rm)}
\leq & C \|\uu(t,\cdot)-\uu(\tau,\cdot)\|_{C^\theta_b({B_R};\Rm)}^{\theta/2}\|\uu(t,\cdot)-\uu(\tau,\cdot)\|_{C^{2+\theta}_b({B_R};\Rm)}^{1-\theta/2} \\
\leq & C_1|t-\tau|^{\theta/2},
\end{align*}
which shows that the function $t\mapsto \uu(t,\cdot)$ belongs to $C^{\theta/2}([s,T];C^2_b({B_R};\Rm))$.
\end{proof}

\section{A remarkable consequence of Proposition \ref{prop:confronto_1}}
\label{section:special_case}

In this section, we prove a relevant consequence of Proposition \ref{prop:confronto_1} in the context of systems of invariant measures. For this purpose, we consider a weakly coupled autonomous version of the operator $\A$ in \eqref{picco}, whose associated semigroup in $C_b(\Rd;\Rm)$ is not positive (i.e., each operator $\T(t)$ does not map the cone of componentwise nonnegative functions into itself). 

For further use, we introduce the operator ${\mathcal A}={\rm Tr}(QD^2)+\langle {\bm b},\nabla\rangle$.
Instead of Hypothesis \ref{hyp-base} we assume the following conditions. 
\begin{hyp}
\label{hyp:special_case}
\begin{enumerate}[\rm (i)]
\item 
The coefficients $q_{ij}=q_{ji}$, $b_j$ and $c_{kh}$ belong to $C^\alpha_{\rm loc}(\Rd)$ for every $i,j=1,\ldots,d$ and $h,k=1,\ldots,m$;
\item
the infimum $\mu_0$ over $\Rd$ of the minimum eigenvalue $\mu^k(x)$ of the matrix $Q(x)=(q_{ij}(x))$ is positive;
\item
$\langle C^P(x)y,y\rangle\le 0$ for every $x\in\Rd$ and $y\in\Rm$, where $C^P$ is defined in Section $\ref{section:preliminaries}$;
\item
there exists a positive function $\varphi\in C^2(\Rd)$, blowing up as $|x|\to\infty$, such that ${\mathcal A}\varphi(x)\le a-c\varphi(x)$ for every $x\in\Rd$ and some positive constants $a,c$;
\item
there exist $\eta\in\Rm\setminus\{0\}$ such that $\eta \in {\rm Ker}(C(x))$ for 
every $x\in\Rd$;
\item
there does not exist a nontrivial set $K\subset \{1, \ldots, m\}$ such that the coefficients $c_{ij}$ identically vanish on $\Rd$ for every $i\in K$ and $j\notin K$.
\end{enumerate}
\end{hyp}

From \cite[Proposition 2.4 \& Theorem 2.6]{DeLo11}, it follows that we can associate a semigroup of bounded operators $\T(t)$ (resp. $\T^P(t)$) to the operator $\A$ (resp. $\A^P$) and
\begin{align}
\label{spacial_case_SP_1}
\max\{\|\T(t)\f(x)\|^2, \|\T^P(t)\f(x)\|^2\}\leq (T(t)\|\f\|^2)(x), \qquad\;\, (t,x)\in[0,\infty)\times \Rd,
\end{align}
for every $\f\in C_b(\Rd;\Rm)$, where
$T(t)$ is the semigroup associated to the scalar operator $\mathcal A$. In particular, $\T(\cdot)\f$ (resp. $\T^P(\cdot)\f$) is the unique classical solution to the Cauchy problem associated to operator $\A$ (resp. $\A^P$), which is bounded in every strip $[0,T]\times\Rd$.

\begin{rmk}
\begin{enumerate}[(i)]
{\rm 
\item
Hypothesis \ref{hyp:special_case}(ii) ensures that also the matrix $C(x)$ is nonpositive for every $x\in \Rd$. Indeed, for every $x_0\in\Rd$ and $y\in\Rm$ it holds that $0\ge \langle C^P(x_0)|y|,|y|\rangle\ge C(x_0)y,y\rangle$,
where $|y|=(|y_1|,\ldots,|y_m|)$;
\item Hypothesis \ref{hyp:special_case}(iv) implies that there exists a unique invariant measure $\mu$ for the scalar semigroup $T(t)$ associated to the scalar operator $\mathcal A$ (see e.g., \cite[Theorem 9.1.20]{Lo17});
\item
differently from \cite[Hypothesis 2.1]{AdAnLo19}, we do not require that the off-diagonal entries of $C$ are nonnegative, and so we do not expect the operator $\T(t)$ ($t\geq0$) to be positive.}
\end{enumerate}
\end{rmk}

\begin{rmk}
\label{rmk:controllo_semigruppo}
{\rm Let $\f\in C_b(\Rd;\Rm)$. Arguing as in the proof of \cite[Theorem 2.6]{DeLo11}, we can show that both $\T(\cdot)\f$ and $\T^P(\cdot)\f$ are the limit in $C^{1+\alpha/2,2+\alpha}(D)$ of the solutions $\uu_n$ to the Cauchy problem with Neumann homogeneous boundary conditions
\begin{align*}
\left\{
\begin{array}{ll}
D_t\vv(t,x)={\bm B}\vv(t,x), & (t,x)\in(0,\infty)\times B_n, \vspace{1mm} \\
\displaystyle \frac{\partial \vv}{\partial \nu}(t,x)={\bm 0}, & (t,x)\in(0,\infty)\times \partial B_n, \vspace{1mm} \\
\vv(0,x)=\f(x), & x\in B_n,
\end{array}
\right.
\end{align*}
with ${\bm B}=\A$ and ${\bm B}=\A^P$, respectively, for every compact subset $D\subset (0,\infty)\times \Rd$. Hence, By applying Proposition \ref{prop:confronto_1} to $\T(t)$ and $\T^P(t)$ (notice that, under Hypothesis \ref{hyp:special_case}, the coefficients of $\A$ satisfy Hypothesis $\ref{hyp-base}$(i)-(iv) and so, in particular, the assumptions of Proposition \ref{prop:confronto_1} are verified) we infer that $|(\T(t)\f)_j(x)|\leq (\T^P(t)|\f|)_j(x)$
for every $\f\in C_b(\Rd;\Rm)$, $j\in\{1,\ldots,m\}$, $t\geq0$ and $x\in\Rd$.}
\end{rmk}

We introduce the set
$\mathcal E^P:=\{\f\in C_b(\Rd;\Rm):\T^P(t)\f=\f \ \forall t\geq0\}$. From the proof of Step 1 in \cite[Proposition 3.2]{AdAnLo19} it follows that $\f\in \mathcal E^P$ if and only if $\f$ is constant and $\f\in \bigcap_{x\in \Rd}{\rm Ker}(C^P(x))$.

A similar characterization holds true for the fixed point of the semigroup $\T(t)$, as the following proposition shows.
\begin{prop}
\label{prop:punti_fissi_T_eta}
$\mathcal E:=\{\f\in C_b(\Rd;\Rm):\T(t)\f=\f \ \forall t\geq0\}={\rm span}\{\eta\}$, where $\eta$ is the vector in Hypothesis $\ref{hyp:special_case}(v)$. Further, the vector $|\eta|$ belongs to $\bigcap_{x\in \Rd}{\rm Ker}(C^P(x))$.
\end{prop}
\begin{proof}
For sake of convenience we split the proof into three steps.

{\it Step 1}. Here, we prove that $\f$ belongs to ${\mathcal E}$ if and only if $\f\equiv\zeta$ for some $\zeta\in\bigcap_{x\in\Rd}{\rm Ker}(C(x))$.
From \eqref{spacial_case_SP_1} it follows that, if $\f\in {\mathcal E}$, then $\|\f\|^2=\|\T(t)\f\|^2\leq T(t)\|\f\|^2$ in $\Rd$ for every $t\in (0,\infty)$.
The above inequality and the invariance property of $\mu$, which yields
$\int_{\Rd}(T(t)\|\f\|^2-\|\f\|^2)d\mu=0$ for every $t>0$, imply that, for every $t>0$, $T(t)\|\f\|^2=\|\f\|^2$ $\mu$- almost everywhere.  Since $\mu$ is equivalent to the Lebesgue measure, this is enough to conclude that  $\|\T(t)\f\|^2=T(t)\|\f\|^2=\|\f\|^2$ in $\Rd$ for every $t>0$. Hence,
$\|\f\|^2$ is a fixed point of the scalar semigroup $T(t)$ and, consequently, is a constant (see \cite[Proposition 9.1.13]{Lo17}). As a byproduct, we can infer that
$0 = D_t\|\f\|^2= D_t\|\T(t)\f\|^2
= 2\langle \T(t)\f,D_t\T(t)\f\rangle 
\le -2\mu_0\|J_x\T(t)\f\|^2$ for every $t>0$,
which shows that $\f\equiv\zeta$ for some $\zeta\in\Rm$ and
${\bm 0}= D_t\f(x)=(D_t\T(t)\f)(x)=(\A\T(t)\f)(x)=\A\f(x)=C(x)\zeta$ for every $x\in\Rd$,
so that $\zeta\in\bigcap_{x\in\Rd}{\rm Ker}(C(x))$.

{\em Step 2.} Here, we fix $\zeta=(\zeta_1,\ldots,\zeta_m)$ in $\bigcap_{x\in\Rd}{\rm Ker}(C(x))$, and prove that $|\zeta|$ belongs to $\mathcal E^P$.
By Step 1, $\zeta$ belongs to $\mathcal E$. Moreover, from Remark \ref{rmk:controllo_semigruppo} we infer that
$|(\T(\cdot)\zeta)_j|\le (\T^P(\cdot)|\zeta|)_j$ in $[0,\infty)\times\Rd$ for every $j=1,\ldots,m$.
By taking \eqref{spacial_case_SP_1} into account we get
$\|\zeta\|^2=\|\T(t)\zeta\|^2\leq \|\T^P(t)|\zeta|\|^2\le T(t)\|\zeta\|^2=\|\zeta\|^2$ for every $t>0$, which gives $\|\zeta\|^2=\|\T(t)\zeta\|^2=\|\T^P(t)|\zeta|\|^2$. 
Hence, $|\zeta_j|=|(\T(t)\zeta)_j|=(\T^P(t)|\zeta|)_j$ for every $t\geq0$, $x\in\R^d$ and $j=1,\ldots,m$, i.e., $|\zeta|$ belongs to $\mathcal E^P$.
Since $\eta\in \bigcap_{x\in\Rd}{\rm Ker}(C(x))$ it follows that $\bm 0\neq |\eta|$ belongs to $\mathcal E^P\equiv \bigcap_{x\in\Rd}{\rm Ker}(C^P(x))$. From \cite[Proposition 3.2]{AdAnLo19}, we infer that 
$\mathcal E^P ={\rm span}\{|\eta|\}=\bigcap_{x\in\Rd}{\rm Ker}(C^P(x))$ and $|\eta_k|> 0$ for every $k=1,\ldots,m$.

{\em Step 3.} Assume by contradiction that there exist two linearly independent vectors $\eta,\zeta\in\mathcal E$. Then, there exist $a\in\R\setminus\{0\}$ and $i,j\in\{1,\ldots,m\}$ such that $(\eta-a\zeta)_i=0$ and $(\eta-a\zeta)_j\neq0$. Since $\eta-a\zeta\in \mathcal E=\bigcap_{x\in\R^d}{\rm Ker}(C(x))$, from Step $2$ we infer that $|\eta-a\zeta|\in \mathcal E^P={\rm span}\{|\eta|\}$, i.e., there exists a non-negative constant $w$ such that $|\eta-a\zeta|=w|\eta|$. Since $|\eta_j-a\zeta_j|\neq0$ it follows that $w\neq0$, and $\eta_i-a\zeta_i=0$ implies that $\eta_i=0$, which contradicts the fact that $|\eta_k|>0$ for every $k=1,\ldots,m$.
\end{proof}

We now recall the definition of systems of invariant measures for vector-valued semigroups. We say that $\boldsymbol\mu=\{\mu_1,\ldots,\mu_m\}$ is a system of invariant measures for $\T(t)$ if $\mu_k$ is a bounded Borel measure on $\Rd$ for every $k=1,\ldots,m$ and
\begin{align}
\label{def_mis_inv}
\sum_{k=1}^m\int_{\Rd}(\T(t)\f)_kd\mu_k=\sum_{k=1}^m\int_{\Rd}f_kd\mu_k, \qquad \f\in C_b(\Rd;\Rm), \quad t\geq0.
\end{align}

By Proposition \ref{prop:punti_fissi_T_eta}, there exists a nontrivial vector $\xi\in \bigcap_{x\in\Rd}{\rm Ker}(C^P(x))$. Hence, the assumptions in \cite{AdAnLo19} are satisfied and we can use the results therein. In particular, in \cite[Proposition 3.2 \& Theorem 3.6]{AdAnLo19} it has been proved that $\mathcal E^P={\rm span}\{\xi\}$, that all the components of the vector $\xi$ have the same sign (here we assume that $\xi_k>0$ for every $k=1,\ldots,m$), and that a family of bounded Borel measures ${\boldsymbol\mu}^P=\{\mu_1^P,\ldots,\mu_m^P\}$ is a system of invariant measures for $\T^P(t)$, i.e., formula \eqref{def_mis_inv} holds true with $\T(t)$ and $\bm\mu$ replaced by $\T^P(t)$ and ${\bm \mu}^P$, respectively,
if and only if there exists a constant $c$ such that $\mu_j=c\xi_j\mu$, where $\mu$ is the unique invariant measure associated to $T(t)$. As a byproduct, for every system of invariant measures $\boldsymbol \mu^P$ for $\T^P(t)$, the semigroup extends to a strongly continuous semigroup (still denoted by $\T^P(t)$) on the space $L^p_{{\boldsymbol \mu}^P}(\Rd;\Rm)$ of functions $\f:\Rd\rightarrow \Rm$ Borel measurable such that $f_j\in L^p_{\mu_j^P}(\Rd)$ for every $j=1,\ldots,m$, 
endowed with the norm
$\|\f\|^p_{L^p_{{\boldsymbol\mu}^p}(\Rd;\R^m)}=\displaystyle\sum_{k=1}^m\int_{\Rd}|f_k|^pd\mu^P_k$.

\begin{rmk}
{\rm Hereafter, we normalize $\eta$ and $\xi$. From the proof of Proposition \ref{prop:punti_fissi_T_eta} it follows that $\zeta=|\eta|$ for every $k=1,\ldots,m$, and so $\eta_k\neq 0$ for every $k=1,\ldots,m$.}
\end{rmk}

The following proposition characterizes the systems of invariant measures for $\T(t)$. 
Its proof is analogous to that of \cite[Theorem 3.6]{AdAnLo19}, with slight changes. Hence, it is omitted.
\begin{prop}
$\bm\mu$ is a system of invariant measures for $\T(t)$ if and only if there exists $c\in\R$ such that $\mu_k=c\eta_k\mu$ for every $k=1,\ldots,m$, where $\mu$ is the invariant measure associated to $T(t)$.
\end{prop}

The measures $\mu_k$, $k=1,\ldots,m$, may not have the same sign, since the components of $\eta$ may have different signs. Hence, when we talk about $L^p$-spaces with respect to $\bm \mu$ we are considering the space of the measurable functions $\f:\Rd\rightarrow \Rm$ such that
$\sum_{k=1}^m\int_{\Rd}|f_k|^pd|\mu_k|<\infty$, where $|\mu_k|$ is the total variation of the measure $\mu_k$, for $k=1,\ldots,m$. We stress that if $\mu_k=c\eta_k\mu$ for some $c\in\R$, then $|\mu_k|=c\xi_k\mu$, i.e., the vector $(|\mu_1|,\ldots,|\mu_m|)$ is a system of invariant measures for $\T^P(t)$. We are able to prove that for every $t\geq0$ the operator $\T(t)$ extends to a bounded linear operator on $L^p_{\bm \mu}(\Rd;\Rm)$.

\begin{prop}
$\T(t)$ extends to a strongly continuous semigroup on $L^p_{\bm \mu}(\Rd;\Rm)$ with $\|\T(t)\|_{L(L^p_{{\boldsymbol\mu}})}\leq 2^{(p-1)/p}$ for every $t\geq0$.
\end{prop}
\begin{proof}
Fix $\f\in C_b(\Rd;\Rm)$. From the computations in \cite[Proposition 3.12]{AdAnLo19} and Remark \ref{rmk:controllo_semigruppo} we infer that
$|(\T(t)\f(x))_k|^p
\leq |(\T^P(t)|\f|(x))_k|^p
\leq 2^{p-1}(\T^P(t)(|f_1|^p,\ldots,|f_m|^p)(x))_k$
for every $x\in\Rd$ and $k=1,\ldots,m$. The invariance of $\widetilde {\boldsymbol\mu}:=(|\mu_1|,\ldots,|\mu_m|)$ with respect to $\T^P(t)$ gives
\begin{align*}
\|\T(t)\f\|^p_{L^p_{{\boldsymbol\mu}}}
\le 2^{p-1}\sum_{k=1}^m\int_{\Rd}(\T^P(t)(|f_1|^p,\ldots,|f_m|^p)(x))_kd\widetilde \mu_k= 2^{p-1}\|\f\|_{L^p_{{\boldsymbol\mu}}}^p.
\end{align*}

The strong continuity of $\T(t)$ follows arguing as in \cite[Proposition 3.12]{AdAnLo19}.
\end{proof}

\subsection{Asymptotic behaviour}
Here, we characterize the asymptotic behaviour of the semigroup $\T(t)$ as $t$ tends to $\infty$. Throughout this subsection, besides Hypotheses \ref{hyp:special_case}, we assume the following additional conditions.
\begin{hyp0}
\label{hyp:asym_beha}
The coefficients of the operator $\A$ belong to $C^{1+\alpha}_{\rm loc}(\Rd)$. Moreover, there exists a constant $C>0$ such that
$\max\{|q_{ij}(x)|,\langle b(x),x\rangle\}\leq C(1+|x|^2)\varphi(x)$ for every $x\in\Rd$ and $i,j=1,\ldots,d$.
\end{hyp0}

The proof of the following proposition is analogous to that of \cite[Proposition 4.3]{AdAnLo19} and so we skip it.

\begin{prop}
\label{prop:conv_L2norm}
For every $\f\in C_c^{3+\alpha}(\Rd;\Rm)$ the $L^2_\mu$-norm of $\|J_x\T(t)\f\|$ vanishes as $t$ tends to $\infty$.
\end{prop}

In \cite[Proposition 2.3]{AdAnLo19} it has been proved that, for every $x\in\Rd$, $0$ is the unique eigenvalue of $C^P(x)$ on the imaginary axis. The proof relies on the Perron-Frobenius theorem for matrices whose entries are nonnegative, and so we cannot directly adapt it to our situation.

\begin{lemm}
\label{lem:eigenvalue_C_x}
For every $x\in\Rd$, the spectrum of the matrix $C(x)$ is contained in the left half-plane, and $0$ is the unique eigenvalue on the imaginary axis.
\end{lemm}
\begin{proof}
At first, we notice that $C(x)$ has not eigenvalues with positive real part, since $C(x)$ is non positive for every $x\in\Rd$.

Next, we fix $x\in\Rd$ and let $\lambda>0$ be such that $\lambda+\mu>0$ for every real eigenvalue $\mu=\mu(x)$ of $C(x)$ and for every real eigenvalue $\mu=\mu(x)$ of $C^P(x)$, and $\lambda+c_{ii}(x)>0$ for every $i=1,\ldots,m$. It follows that $|C(x)+\lambda{\rm Id}|= C^P(x)+\lambda{\rm Id}$, where for a matrix $A=(a_{ij})_{i,j=1}^{m}$ we set  $|A|:=(|a_{ij}|)_{i,j=1}^{m}$. From \cite[Theorem 8.1.18]{HoJo13} we infer that $r(C(x)+\lambda{\rm Id})\leq r(C^P(x)+\lambda{\rm Id})$, where $r(A)$ denotes the spectral radius of the matrix $A$. Since $0$ is an eigenvalue both of $C(x)$ and $C^P(x)$, $\lambda$ is an eigenvalue both of $C(x)+\lambda{\rm Id}$ and $C^P(x)+\lambda {\rm Id}$. This implies that $r(C(x)+\lambda{\rm Id})\geq \lambda$, and from the Perron-Frobenius Theorem, applied to the matrix $C^P(x)+\lambda{\rm Id}$, whose entries are all nonnegative, we get $r(C^P(x)+\lambda{\rm Id})=\lambda$. 
This means that $\lambda=r(C^P(x)+\lambda{\rm Id})\geq r(C(x)+\lambda{\rm Id})\geq \lambda$, which gives $r(C(x)+\lambda{\rm Id})=\lambda$. Hence, $C(x)+\lambda{\rm Id}$ has no eigenvalue of the form $\lambda+i \gamma$, $\gamma\neq0$, which implies that $C(x)$ has no eigenvalues of the form $i\gamma$, $\gamma\neq0$.
\end{proof}

Let $\bm \mu:=\eta \mu$. We set
$\mathcal M_{\f}:=\sum_{i=1}^m\int_{\Rd}f_i d\mu_i$ for every $\f\in B_b(\Rd;\Rm)$. 
We characterize the asymptotic behaviour of $\T(t)$.

\begin{thm}
For every $\f\in B_b(\Rd;\Rm)$ the function $\T(t)\f$ converges to $\mathcal M_{\f}\eta$ locally uniformly in $\Rd$ as $t$ tends to $\infty$. As a byproduct, $\T(t)\f$ converges to $\mathcal M_{\f}\eta$ in $L^p_{\bm \mu}(\Rd;\Rm)$ as $t$ tends to $\infty$, for every $p\in[1,\infty)$.
\end{thm}
\begin{proof}
Proposition \ref{prop:conv_L2norm}, Lemma \ref{lem:eigenvalue_C_x} and \cite[Proposition 2.6]{AdAnLo19} allow us to repeat the proof of \cite[Theorem 4.4]{AdAnLo19} to get the assertion.
\end{proof}

\section{Examples}
In this section, we provide two examples of operators to which the results of this paper apply. The operator $\A$ in the first example satisfies Hypotheses \ref{hyp-base} and \ref{hyp:smoothness_coeff}, so that we can apply Theorem
\ref{teo-Schau-est} to $\A$, while the operator in the latter example enjoys Hypotheses \ref{hyp:special_case} and \ref{hyp:asym_beha}, so that we can apply Theorem \ref{thm-3.4} to this operator.
\begin{example}
{\rm Let the coefficients of the operator $\A$ in \eqref{picco} be defined as follows:
$q_{ij}^k(t,x)=\zeta_{ij}^k(t)(1+|x|^2)^{\alpha_{ij}^k}$, $b_i^k(t,x)=-\eta_i^k(t)x_i(1+|x|^2)^{\beta_i^k}$ and $c_{kh}(t,x)=\theta_{kh}(t)(1+|x|^2)^{\gamma_{kh}}$
for every $(t,x)\in I\times \Rd$, $i,j=1,\ldots,d$ and $h,k=1,\ldots,m$. 
We assume the following additional assumptions.
\begin{enumerate}[\rm (a)]
\item 
the functions $\zeta_{ij}^k=\zeta_{ji}^k$, $\eta_i^k$ and $\theta_{kh}$ belong to $C^{\alpha/2}_{\rm loc}(I)$ for some $\alpha\in(0,1)$, every $i,j=1,\ldots,d$ and $h,k=1,\ldots,m$. Moreover, each function $\eta_i^k$ is positive in $I$, and $\alpha_{ij}^k=\alpha_{ji}^k$, $\beta_i^k$ and $\gamma_{kh}$ are nonnegative numbers;
\item 
$\theta_{kk}<0$ in $I$ and $\gamma_{kh}<\gamma_{kk}$ 
for every $h,k=1,\ldots,m$, and there exists no nontrivial sets $K\subset \{1,\ldots,m\}$ such that $\theta_{kh}$ identically vanishes on $I$ for every $k\in K$ and $h\notin K$;
\item 
for every $k=1,\ldots,m$ it holds that $\alpha_{\min}^k:=\min_{i=1,\dots,d} \alpha_{ii}^k\geq \max_{i\neq j}\alpha_{ij}^k$ and
\begin{align}
\label{exa_cond_ell}
\min_{i=1,\ldots,d}\zeta_{ii}^k-\max_{i=1,\ldots,d}\bigg (\sum_{j\neq i}|\zeta_{ij}^k|^2\bigg )^{\frac{1}{2}}>0\quad {\rm in~} I;
\end{align}
\item 
for every $k=1\ldots,m$ it holds that $\max_{i=1,\ldots,d}\alpha_{ii}^k\leq 1+\max_{i=1,\ldots,d}\{\gamma_{kk}, \beta_i^k\}$.
\end{enumerate}

Under these assumptions, Hypotheses $\ref{hyp-base}$ are satisfied. Indeed, Hypothesis \ref{hyp-base}(i) is immediately verified, and, as far as Hypothesis \ref{hyp-base}(ii) is concerned, we fix a bounded set $J\subset I$ and observe that
\begin{align*}
& \langle Q^k(t,x)\xi,\xi\rangle
\geq (1+|x|^2)^{\alpha_{\min}^k}\bigg [\min_{i=1,\ldots,d}\zeta_{ii}^k(t)-\max_{i=1,\ldots,d}\bigg (\sum_{j\neq i}|\zeta_{ij}^k(t)|^2\bigg )^{\frac{1}{2}}\bigg ]|\xi|^2=:\mu^k(t,x)|\xi|^2
\end{align*}
for every $(t,x)\in J\times\Rd$, $k=1,\ldots,m$ and $\xi\in\Rd$, and condition \eqref{exa_cond_ell} implies that the infimum over $J\times\Rd$ of $\mu^k$ is positive. Hypothesis \ref{hyp-base}(iii) is satisfied by the function $\boldsymbol\varphi:\Rd\to\R^m$, defined by  $\boldsymbol\varphi(x):=(1+|x|^2)\1$ for every $x\in\Rd$,
due to conditions (a)-(c).
Finally, Hypotheses \ref{hyp-base}(iv)-(v) follow straightforwardly from condition (b).

We now introduce the following more restrictive conditions on $\eta_i^k$ and $\beta_i^k$.

\begin{enumerate}[\rm (a$^\prime$)]
\item
$\eta_i^k=\eta^k$ and $\beta_i^k=\beta^k$ for every $i$ and some $\eta^k$ and $\beta^k$;
\item 
$\alpha_{ij}^k\leq \alpha_{\min}^k+1/2$ for every $i,j=1,\ldots,d$.
\end{enumerate}

Under these additional assumptions, also Hypotheses \ref{hyp:smoothness_coeff} are satisfied. We just need to verify the last three assumptions in Hypotheses \ref{hyp:smoothness_coeff}. To begin with, we observe that $r^k(t,x)=-\eta^k(t)(1+|x|^2)^{\beta^k}$ for every $(t,x)\in [s,T]\times \Rd$ and $k=1,\ldots,m$.
To prove that Hypothesis \ref{hyp:smoothness_coeff}(iii) is satisfied, it is enough to observe that
\begin{align*}
|{\rm Tr}(Q^{k}(t,x))|+|(Q^{k}(t,x)x)_i|\leq C(1+|x|^2)^{\max_{j=1,\ldots,d}\alpha_{ij}^k+1/2}
\end{align*}
for some positive constant $C$, every $x\in\Rd$ and $i=1,\ldots d$, and $\mu^k(t,x)\sim(1+|x|^2)^{\alpha_{\min}^k}$ as $|x|$ tends to $\infty$ for every $k=1,\ldots,m$. Hence, condition (b$^{\prime}$) allows us to conclude.

As far Hypothesis \ref{hyp:smoothness_coeff}(iv) is concerned, we observe that for every $k=1,\ldots,m$, $\ell=2,3$ and $s=1,2,3$
it holds that $b^{k,\ell}\sim (1+|x|^2)^{\beta^k+(1-\ell)/2}$ and $q^{k,s}\sim(1+|x|^2)^{\max_{i,j}\alpha_{ij}^k-s/2}$ as $|x|$ tends to $\infty$. Hence, Hypothesis \ref{hyp:smoothness_coeff}(iv) is fulfilled if
$\max_{i,j}\alpha_{ij}^k-1/2\leq \alpha_{\min}^k$
and these inequalities are verified under the conditions (a$^{\prime}$) and (b$^{\prime}$). Finally, since
$|M_{k}(t,x)|\sim(1+|x|^2)^{\gamma_{kk}}$ and $c^{k,s}\sim(1+|x|^2)^{\gamma_{kk}-s/2}$ as $|x|$ tends to $\infty$, for $s=1,2,3$,
it follows that also Hypothesis \ref{hyp:smoothness_coeff}(v) is satisfied.}
\end{example}

\begin{example}
{\rm Let us assume that the coefficients of the autonomous operator $\A$ are given by $q_{ij}(x)=\zeta_{ij}(1+|x|^2)^{\alpha}$,  $b_i(x)=-\theta_ix_i(1+|x|^2)^{\beta}$ and
\begin{align*}
 C(x)=(1+|x|^2)^\gamma\left(
\begin{matrix}
-1 & 0 & - 1\\
0 & -3 & \sqrt 3 \\
-1 & \sqrt 3 & -2
\end{matrix}
\right)
\end{align*}
for every $x\in\Rd$, where
the matrix $R=(\zeta_{ij})$ is symmetric and positive definite,
$\theta_i$ are all positive constants and $\alpha$, $\beta$, $\gamma$ are nonnegative constants, with $\alpha<\beta+1$.

If we consider the function $\varphi:\Rd\to\R$, defined by $\varphi(x)=(1+|x|^2)^k$ for every $x\in\Rd$ and some $k=\max\{\alpha-1,1\}$, then Hypotheses \ref{hyp:special_case} and \ref{hyp:asym_beha} are satisfied. Finally, we observe that the eigenvalues of $C$ and of $C^P$ are $0, -3+\sqrt 2$ and $-3-\sqrt 2$, $\eta=(-\sqrt 3,1,\sqrt 3)$ and $\xi=(\sqrt 3,1,\sqrt 3)$.}
\end{example}

\appendix
\section{}

Here, we collect some useful a priori estimates on $\G(t,s)\f $ and on $\T_t(\tau)\f$ that we use in the paper.

\begin{prop}
\label{prop:app_stime_varie}
Let $\Omega\subset \Rd$ be a bounded open set and let the coefficients of the operator $\A$ in \eqref{picco} be $\alpha$-H\"older continuous in $[a,b]\times\Omega$ for some $\alpha\in (0,1)$ and $a,b\in I$, with $a<b$. Further, assume that there exists a positive constant $\mu_0$ such that $\langle Q^k(t,x)\xi,\xi\rangle\ge \mu_0|\xi|^2$ for every $(t,x)\in [a,b]\times\Omega$ and $\xi\in\Rd$. Finally, let $B_{R_1}(x_0)$ be a ball compactly contained in $\Omega$. Then, the following properties are satisfied.
\begin{enumerate}[\rm (i)]
\item 
Let $\uu\in C_b([a,b]\times \Omega;\Rm)\cap C^{1,2}((a, b)\times \Omega;\Rm)$ satisfy the equation $D_t\uu=\A\uu+\g$ in $(a, b)\times \Omega$, for some $\g\in C^{\alpha/2,\alpha}((a,b)\times \Omega;\Rm)$.  
Then, there exists a positive constant $C$ which depends on $R_1$, $\mu_0$, $d$, $m$, $a$, $b$ and $\sup_{r\in[a,b]}\|\A(r)\|_{\alpha,\Omega}$ such that
\begin{align*}
(t-a)^{\frac{j}{2}}\|\uu(t,\cdot)\|_{C^j_b(B_{R_1}(x_0);\Rm)}\le C\big (\|\uu\|_{C_b((a,b)\times {\Omega};\Rm)}+\|\g\|_{C^{\alpha/2,\alpha}((a,b)\times \Omega;\Rm)}\big ).
\end{align*}
for $j=1,2$ and $t\in(a,b)$.
\item 
In addition to the assumptions in $(i)$, assume that $\uu \in C^{1+\alpha/2,2+\alpha}_{\rm loc}((a, b]\times \Omega;\Rm)$. Then, for every $\delta\in(0,(b-a)/2)$ and every pair of bounded open sets $\Omega_1,\Omega _2$, such that $\Omega_1\Subset \Omega_2\Subset \Omega$, there exists a positive constant $K$, depending on $\Omega_1,\Omega_2,\delta,a,b$ and $\sup_{r\in[a,b]}\|\A(r)\|_{\alpha,\Omega}$, such that
\begin{align*}
\;\;\;\;\;\;\|\uu\|_{C^{1+\alpha/2,2+\alpha}((a+2\delta,b)\times \Omega_1;\Rm)}
\!\leq\! K\big ( \|\uu\|_{C_b((a+\delta,b)\times \Omega_2;\Rm)}\!+\!\|\g\|_{C^{\alpha/2,\alpha}((a+\delta,b)\times \Omega_2;\Rm)}\big).
\end{align*}
\item 
Fix $r\in[a,b]$ and $\mathcal T>0$, and let $\uu$ be a function in $C_b([0,\mathcal T]\times \Omega;\Rm)\cap C^{1,2}((0,\mathcal T)\times \Omega;\Rm)$ such that  $D_t\uu=\A(r)\uu+\g$ in $(0,\mathcal T)\times \Omega$, where $\g\in C^{\alpha/2,\alpha}((0,\mathcal T)\times \Omega;\Rm)$. Then, 
there exists a positive constant $C$ which depends on $R_1$, $\mu_0$, $d$, $m$, $\sup_{r\in[a,b]}\|\A(r)\|_{\alpha,\Omega}$ and on $[a,b]$,
but is independent of $r$, such that
\begin{align}
t^{\frac{j}{2}}\|\uu(t,\cdot)\|_{C^j_b(B_{R_1}(x_0);\Rm)}
\leq & C\big (\|\uu\|_{C_b((0,\mathcal T)\times {\Omega};\Rm)}+\|\g\|_{C^{\alpha/2,\alpha}((0,\mathcal T)\times \Omega;\Rm)}\big ),
\label{app_stima_finale_1}
\end{align}
for $j=1,2$ and $t\in(0,\mathcal T)$.
\item 
Fix $r\in[a,b]$ and $\mathcal T>0$, and let $\uu\in C_{\rm loc}^{1+\alpha/2,2+\alpha}((0,\mathcal T]\times \Omega;\Rm)$ satisfy the equation $D_t\uu=\A(r)\uu+\g$ in $(0,\mathcal T)\times \Omega$, with $\g\in C^{\alpha/2,\alpha}_{\rm loc}((0,\mathcal T]\times \Omega;\Rm)$. Then, for every $\delta\in (0,\mathcal T)$ and every pair of bounded open sets $\Omega_1,\Omega _2$, such that $\Omega_1\Subset \Omega_2\Subset \Omega$, there exists a positive constant $K$, depending on $\Omega_1$, $\Omega_2$, $\delta$, $\mathcal T$ and $\sup_{r\in[a,b]}\|\A(r)\|_{\alpha,\Omega}$, and independent of $r$, such that
\begin{align}
\;\;\;\;\;\;\|\uu\|_{C^{1+\alpha/2,2+\alpha}((2\delta,\mathcal T)\times \Omega_1;\Rm)}\leq K\big (\|\uu\|_{C_b((\delta,\mathcal T)\times \Omega_2;\Rm)}+\|\g\|_{C^{\alpha/2,\alpha}((\delta,\mathcal T)\times \Omega_2;\Rm)}\big ).
\label{int_sch_est_app}
\end{align}
\end{enumerate}
\end{prop}
\begin{proof}
Properties (i) and (ii) can be proved, with slight modifications, arguing as in the proofs of \cite[Proposition A.1 \& Theorem A.2]{AALT17}; whereas properties (iii) and (iv) are a byproduct of (i) and (ii). Let us notice that we can choose the constants in \eqref{app_stima_finale_1} and \eqref{int_sch_est_app} to be uniform with respect to $r\in [a,b]$ since the constants in (i) and (ii) are uniform in $(a,b)$.
\end{proof}

\end{document}